\documentclass{amsart}

\usepackage{amsthm,amsfonts,amsmath,amssymb,latexsym,epsfig,mathrsfs,yfonts,marvosym}
\usepackage[usenames]{color}

\DeclareMathAlphabet\oldmathcal{OMS}        {cmsy}{b}{n}
\SetMathAlphabet    \oldmathcal{normal}{OMS}{cmsy}{m}{n}
\DeclareMathAlphabet\oldmathbcal{OMS}       {cmsy}{b}{n}
 
\usepackage{eucal}

\usepackage{graphicx}
\usepackage[all]{xy}
\usepackage{epsfig}
\usepackage{layout}

\newtheorem{theorem}{Theorem}[section]
\newtheorem{lemma}[theorem]{Lemma}
\newtheorem{proposition}[theorem]{Proposition}
\newtheorem{corollary}[theorem]{Corollary}
\newtheorem{definition}[theorem]{Definition}
\newtheorem{definition/proposition}[theorem]{Definition/Proposition}

\newenvironment{example}{\medskip \refstepcounter{theorem}
\noindent  {\bf Example \thetheorem}.\rm}{\,}
\newenvironment{remark}{\medskip \refstepcounter{theorem}
\noindent  {\bf Remark \thetheorem}.\rm}{\,}
\newtheorem{ack}{Acknowledgments} 

\def\BOne{{\mathchoice {\rm 1\mskip-4mu l} {\rm 1\mskip-4mu l}
                          {\rm 1\mskip-4.5mu l} {\rm 1\mskip-5mu l}}}

\def\fract#1#2{\raise4pt\hbox{$ #1 \atop #2 $}}

\def\lcm{{\rm lcm}}

\def\bbc{{\mathbb C}}

\def\bbp{{\mathbb P}}
\def\bbq{{\mathbb Q}}
\def\bbr{{\mathbb R}}

\def\bbz{{\mathbb Z}}

\def\gra{\alpha}

\def\gre{\epsilon}

\def\grg{\gamma}
\def\gri{\iota}
\def\grk{\kappa}
\def\grl{\lambda}

\def\gro{\omega}

\def\grr{\rho}

\def\grt{\tau}

\def\grz{\zeta}
\def\grD{\Delta}

\def\grF{\Phi}
\def\grG{\Gamma}

\def\bfl{{\bf l}}

\def\bfp{{\bf p}}

\def\bfu{{\bf u}}

\def\bfw{{\bf w}}
\def\bfx{{\bf x}}

\def\bfR{{\boldsymbol R}}
\def\bfS{{\boldsymbol S}}

\def\cala{{\mathcal A}}
\def\calb{{\mathcal B}}
\def\calc{{\mathcal C}}
\def\calo{{\mathcal O}}
\def\calu{{\mathcal U}}
\def\cald{{\mathcal D}}

\def\calf{{\mathcal F}}

\def\cali{{\mathcal I}}
\def\calj{{\mathcal J}}
\def\calk{{\mathcal K}}

\def\calm{{\mathcal M}}
\def\caln{{\mathcal N}}
\def\calo{{\mathcal O}}

\def\calr{{\mathcal R}}
\def\cals{{\oldmathcal S}}

\def\calu{{\mathcal U}}
\def\calv{{\mathcal V}}
\def\calw{{\mathcal W}}
\def\caly{{\mathcal Y}}
\def\calz{{\oldmathcal Z}}
\def\calx{{\mathcal X}}

\def\la#1{\hbox to #1pc{\leftarrowfill}}
\def\ra#1{\hbox to #1pc{\rightarrowfill}}

\def\ga{{\mathfrak a}}

\def\gc{{\mathfrak c}}

\def\gf{{\mathfrak f}}
\def\gg{{\mathfrak g}}
\def\gh{{\mathfrak h}}
\def\gi{{\mathfrak i}}

\def\gl{{\mathfrak l}}
\def\gm{{\mathfrak m}}
\def\gn{{\mathfrak n}}
\def\go{{\mathfrak o}}

\def\gr{{\mathfrak r}}
\def\gs{{\mathfrak s}}
\def\gt{{\mathfrak t}}
\def\gu{{\mathfrak u}}

\def\gy{{\mathfrak y}}
\def\gz{{\mathfrak z}}

\def\gA{{\mathfrak A}}
\def\gB{{\mathfrak B}}
\def\gC{{\mathfrak C}}
\def\gD{{\mathfrak D}}

\def\gF{{\mathfrak F}}
\def\gG{{\mathfrak G}}
\def\gH{{\mathfrak H}}

\def\gQ{{\mathfrak Q}}
\def\gR{{\mathfrak R}}
\def\gS{{\mathfrak S}}
\def\gT{{\mathfrak T}}

\def\X{\frak{X}}

\def\tU{\tilde{{U}}}
\def\tV{\tilde{{V}}}

\def\tX{\tilde{{X}}}
\def\tf{\tilde{{f}}}

\def\tx{\tilde{{x}}}

\def\tx{\tilde{x}}

\def\bz{\bar{z}}

\def\cJ{\hat{J}}

\def\tcald{\tilde{\cald}}

\def\tR{\tilde{R}}
\def\tV{\tilde{V}}

\def\hook{\mathbin{\hbox to 6pt{%
                 \vrule height0.4pt width5pt depth0pt
                 \kern-.4pt
                 \vrule height6pt width0.4pt depth0pt\hss}}}

\begin{document}
\bibliographystyle{amsalpha}

\title{Maximal Tori in Contactomorphism Groups}\thanks{During the beginning of this work the author was partially supported by NSF grant DMS-0504367.}

\author{Charles P. Boyer}
\address{Department of Mathematics and Statistics,
University of New Mexico, Albuquerque, NM 87131.}

\email{cboyer@math.unm.edu} 



\begin{abstract}
I describe a general scheme which associates conjugacy classes of tori in the contactomorphism group to transverse almost complex structures on a compact contact manifold. Moreover, to tori of Reeb type whose Lie algebra contains a Reeb vector field one can associate a Sasaki cone. Thus, for contact structures $\cald$ of K-contact type one obtains a configuration of Sasaki cones called a {\it bouquet} such that each Sasaki cone is associated with a conjugacy class of tori of Reeb type. 
\end{abstract}

\maketitle

%\pagestyle{myheadings}
%\markboth{Maximal Tori in Contactomorphism Groups}{Charles P. Boyer}
\tableofcontents

\section{Introduction}

The main purpose of this paper is to study the relationship between compatible almost complex structures on a contact manifold and conjugacy classes of maximal tori in the contactomorphism group. 
It draws its inspiration from the work of Karshon \cite{Kar03} and Lerman \cite{Ler03b}. The former gives a count of the conjugacy classes of maximal tori in the symplectomorphism group of $S^2\times S^2$ with symplectic structures $\gro_{k_1,k_2}$ depending on a pair of positive integers, and illustrates the close relation with the well-known even Hirzebruch surfaces. A similar result is given in \cite{Kar03} for symplectic structures on the non-trivial $S^2$-bundle over $S^2$ which relate to the odd Hirzebruch surfaces. This work in turn was motivated by Gromov's seminal paper \cite{Gro85} where it is observed that the topology of the  symplectomorphism groups on $S^2\times S^2$ changes as the pair of integers change.

In \cite{Ler03b} Lerman uses Karshon's results to give a lower bound on the number of conjugacy classes of maximal tori in certain contact structures on $S^2\times S^3$. While these contact structures are actually quite special, they are Sasakian, Lerman had previously shown in \cite{Ler01} that a count can be made in other, non-Sasakian cases. In particular, he showed that the overtwisted contact structure on $S^3$ admits a countable infinity of conjugacy classes of maximal tori. This has led the present author to try to construct a general theory relating compatible transverse almost complex structures on contact manifolds to conjugacy classes of maximal tori. This has become possible due to a recent result of Frances \cite{Fra07} that generalizes to the `almost' category a well-known result of Lee and Schoen \cite{Lee96,Sch95} that on a compact manifold $M$ with a strictly pseudoconvex CR structure, the group of CR transformations is compact except when $M$ is an odd dimensional sphere $S^{2n+1}$ with its standard CR structure (see Theorem \ref{Schthm} below). The latter is well understood, its group of CR transformations being $SU(n+1,1)$ \cite{Web77}.

In this paper our contact manifolds are always compact. At this stage our general framework does not work in the non-compact case where little seems to be known. One exception seems to be the Heisenberg group \cite{Boy09}. Most of the results and examples in this paper concern contact structures of K-contact or Sasaki type, since these are generally by far the most tractible. Such contact structures have contact 1-forms that fibre in the orbifold sense over a compact symplectic orbifold with cyclic local uniformizing groups. Thus, it is necessary to generalize certain results that are known for symplectic manifolds and their symplectomorphism groups to the case of symplectic orbifolds. In particular, I generalize to any quasi-regular contact structure a result of Banyaga \cite{Ban78} that says that for a regular contact structure the subgroup of the contactomorphism group that leaves invariant a regular contact form is a central $S^1$ extension of the group of Hamiltonian isotopies of the base symplectic manifold, and subsequently generalize a result of Lerman \cite{Ler02b} saying that a maximal torus of the group of Hamiltonian isotopies induces a unique maximal torus in the corresponding contactomorphism group.

The paper is organized as follows: Section \ref{basdef} gives a review of the essential structures that we are dealing with and their interrelation. Section \ref{congroupsect} discusses the contactomorphism group and some important subgroups. In particular, it is shown that the subgroup leaving a contact 1-form invariant is a Fr\'echet Lie subgroup of the full contactomorphism group. Section \ref{secalmostcomp} defines the fundamental map between compatible almost complex structures and conjugacy classes of maximal tori and shows that it is surjective. It also defines the Sasaki bouquet which gives a natural configuration of Sasakian structures associated with certain contact structures. In Section \ref{symgroup} I generalize known results about the symplectomorphism group in the category of manifolds to that of cyclic orbifolds, and in Section \ref{consymgr}  I give three main theorems relating the Hamiltonian symplectomorphism 
group to the contactomorphism group in the quasi-regular case. Section \ref{torconsect} begins with a review of toric contact structures, and then proves that the fundamental map relating transverse almost complex structures to conjugacy classes of maximal tori is bijective in the toric case when restricting to toric contact structures of Reeb type and maximal tori of Reeb type. Section \ref{torconsect} ends by describing many examples in dimensions three and five, some of which give non-trivial Sasaki bouquets; while Section \ref{moreex} mainly presents results in higher dimensions and with non-toric contact structures. More specific examples of Sasaki bouquets have now been given in dimension five \cite{BoTo11,BoTo12}.

\begin{ack}
{\rm At various times during this work I have benefited from conversations and/or emails from V. Apostolov, A. Banyaga, A. Buium, P. Gauduchon, P. Michor, J. Pati, Y.-S. Poon, C. T{\o}nnesen-Friedman, and D. Vassilev.}
\end{ack}

\section{Basic Structures}\label{basdef}

\subsection{Contact Structures}
Recall that a {\it contact structure} on a $2n+1$ dimensional smooth manifold $M$ is a maximally non-integrable codimension one subbundle $\cald$ of the tangent bundle $TM.$ We shall assume that $M$ is orientable with a fixed orientation unless stated to the contrary. We can think of the subbundle $\cald$ as being defined by an equivalence class of (contact) 1-forms $\eta$ satisfying $\eta\wedge (d\eta)^n\neq 0$ where $\eta'\sim \eta$ if there is a nowhere vanishing function $f\in C^\infty(M)$ such that $\eta'=f\eta$ and $\cald=\ker\eta.$ We shall also choose a co-orientation by choosing a 1-form $\eta$ and then restricting the functions $f$ to strictly positive functions $C^\infty(M)^+.$ With the co-orientation fixed we denote by $\gC^+(\cald)$ the set of contact 1-forms representing $\cald$ with its orientation fixed. Choosing a 1-form $\eta\in \gC^+(\cald)$ gives an identification of 
$\gC^+(\cald)$ with $C^\infty(M)^+.$ Note that when $n$ is even changing $\eta$ to $-\eta$ just reverses the orientation of $M$; however, when $n$ is odd it does not effect the orientation of $M$.

Given a contact 1-form $\eta$ there is a unique vector field $\xi$ called the {\it Reeb vector field} that satisfies $\eta(\xi)=1$ and $\xi\hook d\eta=0.$ Thus, there is a bijection between the set $\gC^+(\cald)$ and a subset $\calr^+(\cald)$ of the Lie algebra of infinitesimal contact transformations consisting of the Reeb vector fields of any contact form representing the oriented contact structure $\cald$. It is easy to see \cite{Boy08} that when $M$ is compact $\calr^+(\cald)$ is an open convex cone in $\X(M)$ in the subspace topology, that I call the {\it Reeb cone}. Now $\xi$ generates a trivial real line bundle $L_\xi$ and the {\it characteristic foliation} $\calf_\xi.$ This provides a splitting of the tangent bundle $TM=\cald \oplus L_{\xi}$ and fixes a symplectic structure $d\eta|_\cald$ in the vector bundle $\cald.$ Changing the contact form $\eta\mapsto f\eta$ changes the symplectic structure by a conformal factor, $d\eta|_\cald\mapsto d\eta'|_\cald= fd\eta|_\cald,$ and it also changes the splitting $\cald\oplus L_\xi\mapsto \cald\oplus L_{\xi'}$.

Recall that a complex structure $J$ on a vector bundle $\cald$ is a smooth endomorphism such that $J^2=-\BOne$.
We are interested in the set of all complex structures $J$ on $\cald$ that satisfy a certain compatibility condition. 
 
\begin{definition}\label{compatibleJ}
Choose a smooth contact form $\eta\in\gC^+(\cald)$ representing $\cald$. We say that a complex structure $J$ is {\bf compatible} with $\cald$ if for any sections $X,Y$ of $\cald$ we have\footnote{I use the convention used in \cite{BG05} which is opposite to the convention most frequently used in K\"ahler geometry.}
$$d\eta(JX,JY)=d\eta(X,Y), \qquad d\eta(JX,Y)>0.$$
We denote by $\calj(\cald)$ the set of all complex structures compatible with $\cald.$
\end{definition}

This definition is well-defined since

\begin{proposition}\label{compwd}
The definition of compatible in Definition \ref{compatibleJ} is independent of the choice of 1-form $\eta$ representing $\cald$ as an oriented contact structure.
\end{proposition}

\begin{proof}
Let $\eta'$ be any other contact form in $\gC^+(\cald)$. Then $\eta'=f\eta$ for some $f\in C^\infty(M)^+.$ So
$$d\eta'(JX,JY)=fd\eta(JX,JY)=fd\eta(X,Y)=d\eta'(X,Y)$$
since $X,Y$ are sections of $\cald.$
\end{proof}

\subsection{Almost CR Structures and CR Structures}\label{almostCR}
Let $\cald$ be a codimension one subbundle of the tangent $TM$ of a $2n+1$ dimensional smooth manifold $M$. A (codimension one) almost CR structure is given by a splitting of the complexified bundle $\cald\otimes \bbc=\cald^{1,0} +\cald^{0,1}$ such that $\cald^{1,0}\cap\cald^{0,1}=\{0\}$ and $\cald^{0,1}=\overline{\cald^{1,0}},$ where the bar denotes complex conjugation \cite{DrTo06}. The almost CR structure is said to be {\it (formally) integrable}\footnote{The word ``formally'' appears in parenthesis since for general almost CR structures this definition of integrablility is not equivalent to the existence of a transverse complex coordinate system. However, it is for those related to contact structures, so I will not use the terminology ``formally integrable'', but just ``integrable''.} if $[\grG(\cald^{1,0}),\grG(\cald^{1,0})]\subset \grG(\cald^{1,0}).$

An alternative definition of an almost CR structure, and the point of view adopted here, can be given by the existence of a codimension one vector bundle $\cald$ of $TM$ together with a smooth endomorphism $J$ of $\cald$ such that $J^2=-\BOne$. Then the almost CR structure is said to be {\it integrable} if it satisfies two conditions:
\begin{enumerate}
\item $[JX,Y]+[X,JY]\in \grG(\cald)$ whenever $X,Y\in \grG(\cald).$
\item $N(X,Y)=[JX,JY] -[X,Y]-J[JX,Y]-J[X,JY]=0.$
\end{enumerate}
A less well-known definition that is important for us is: an almost CR structure $(\cald,J)$ is said to be {\it partially integrable} if condition (1) holds. We can rephrase condition (1) as follows: let $\cald$ be given as the kernel of a 1-form, i.e. $\cald=\ker\eta,$ and define $g_\cald(X,Y)=d\eta(JX,Y)$. The 2-form $d\eta$ is called the {\it Levi form} of the CR structure. One easily sees that condition (1) is equivalent to $g_\cald(X,Y)=g_\cald(Y,X).$ Moreover, since $J^2=-\BOne$ we see that (1) implies $g_\cald(JX,JY)=g_\cald(X,Y),$ that is, a partially integrable CR structure gives rise to a symmetric 2-tensor field $g_\cald$ that is invariant under $J.$  We are interested in the case that $g_\cald$ (or $-g_\cald$) defines an Hermitian metric on the vector bundle $\cald.$ When this happens the CR structure is called {\it strictly pseudoconvex}. Summarizing we have

\begin{proposition}\label{CRcon}
Let $\cald$ be a contact structure on a smooth orientable manifold $M$, and choose a compatible almost complex structure $J\in\calj(\cald).$ Then the pair $(\cald,J)$ defines a partially integrable strictly pseudoconvex almost CR structure.
\end{proposition}

Of particular interest is the case when the almost CR structure is integrable, that is both conditions (1) and (2) hold. We denote by $\calj_I(\cald)$ the subset of $\calj(\cald)$ consisting of integrable almost CR structures.

\subsection{Contact Metric Structures}
For contact structures we can think of the choice of $\eta\in \gC^+(\cald)$ as {\it polarizing} the almost CR structure $(\cald,J)$. Accordingly we refer to the pair $(\eta,J)$ as a {\it polarized almost CR structure}. 
It is convenient to extend the almost complex structure $J$ to an endomorphism $\Phi$ by setting $J=\Phi|_\cald,~\Phi\xi=0.$ This extension depends on the choice of $\eta\in\gC^+(\cald),$ and we can think of the pair $(\eta,\Phi)$ as a polarized almost CR structure. We say that a Riemannian metric $g$ is {\it compatible} with $(\eta,\Phi)$, or equivalently with $(\eta,J)$, if 
$$g(\Phi X,\Phi Y)=g(X,Y) -\eta(X)\eta(Y)$$
for any vector fields $X,Y.$ Any such metric satisfies $g(\xi,\xi)=1$ and $g(X,\xi)=0$ for $X$ a section of $\cald.$ That is, for any metric $g$ that is compatible with $(\eta,\Phi)$ the Reeb vector field has norm one and is orthogonal to $\cald.$ So there is a 1-1 correspondence between Riemannian metrics that are compatible with $(\eta,\Phi)$ and Hermitian metrics on the vector bundle $\cald$ that are compatible with $J.$ Note that $g_\cald(X,Y)=d\eta(JX,Y)$ defines a natural metric in the vector bunde $\cald$ where here $X,Y$ are sections of $\cald$. So there is a canonical Riemannian metric compatible with the pair $(\eta,\Phi),$ namely  
\begin{equation}\label{conmet}
g=g_\cald\oplus \eta\otimes \eta=d\eta\circ(\Phi \otimes \BOne)+\eta\otimes \eta.
\end{equation}
We call the quadruple $(\eta,\xi,\Phi,g)$ where $g$ has the form (\ref{conmet}) a {\it contact metric structure}, and denote by $\calc\calm(\cald)$ the set of all such contact metric structures whose underlying contact structure is $\cald.$ From this point of view we have
\begin{proposition}\label{canconmet}
Fixing $\eta$ there is a 1-1 correspondence between compatible almost complex structures on $\cald$ and canonical contact metrics of the form {\rm (\ref{conmet})}.
\end{proposition}

\subsection{K-contact and Sasakian Structures}
The Reeb vector field $\xi$ leaves invariant the contact 1-form $\eta;$ hence, it leaves the contact structure $\cald$ invariant. However, generally it does not leave the almost CR structure $J$ invariant; when it does, the contact structure is called {\it K-contact}. In this case the metric $g$ given by Equation (\ref{conmet}) is called {\it bundle-like}. In fact, $g$ in Equation (\ref{conmet}) is bundle-like and gives $\calf_\xi$ the structure of a Riemannian foliation if and only if the almost CR structure is invariant under the flow of the Reeb vector field, or equivalently the metric is invariant under the Reeb flow. In equations this is equivalent to $\pounds_\xi\Phi=0=\pounds_\xi g.$ So in the K-contact case $\xi$ is a Killing vector field giving rise to the name K-contact. 

\begin{definition}\label{kcontype}
We say that a contact structure $\cald$ on $M$ is of {\bf K-contact type} if there is a 1-form $\eta$ in the contact structure and a choice of almost complex structure $J$ such that $\cals=(\xi,\eta,\Phi,g)$ is K-contact. In addition, if the underlying almost CR structure $(\cald,J)$ is integrable we say that $\cald$ is of {\bf Sasaki type}. 
\end{definition}

K-contact structures are closely related to the quasi-regularity of the characteristic foliation $\calf_\xi$ of a strict contact structure. Recall that a strict contact structure $(M,\eta)$ is a contact structure $(M,\cald)$ with a fixed contact form $\eta$ satisfying $\ker\eta=\cald$, and it is said to be {\it quasi-regular}\footnote{Quasi-regular is sometimes called {\it almost regular} in the literature} if there is a positive integer $k$ such that each point has a
foliated coordinate chart $(U,x)$ such that each leaf of
$\calf_\xi$ passes through $U$ at most $k$ times. If $k=1$ then
the foliation is called {\it regular}. A strict contact structure that is not quasi-regular is called {\it irregular}. I often apply the definitions of K-contact and quasi-regular, etc.  to the Reeb field $\xi$ or the contact form $\eta.$  We have \cite{BG05}

\begin{proposition}\label{qkcon}
Let $(M,\eta)$ be a quasi-regular strict contact manifold such that the characteristic foliation $\calf_\xi$ has all compact leaves. Then then there exists a choice of Riemannian metric $g$ (and thus $\Phi$) such that $(\xi,\eta,\Phi,g)$ is K-contact. In particular, a quasi-regular contact structure on a compact manifold has a compatible K-contact structure.
\end{proposition}

Not all K-contact structures are quasi-regular; however, a result of Rukimbira \cite{Ruk95a} states that every K-contact structure can be approximated by quasi-regular ones whose Reeb vector fields lie in the Sasaki cone (see Section \ref{secSascone}) of the original K-contact structure making irregular K-contact structures quite tractible. In particular, for a K-contact structure the foliation $\calf_\xi$ is Riemannian. 
We are interested in 
\begin{equation}\label{allKcon}
\gR^+(\cald)=\{\xi\in \calr^+(\cald) ~|~\xi ~\text{is ~K-contact} \}.
\end{equation}

We shall make much use of what has been become known as the {\it orbifold Boothby-Wang construction} \cite{BoWa,BG00a,BG05}:

\begin{theorem}\label{bgthm}
Let $(M,\xi,\eta,\Phi,g)$ be a quasi-regular K-contact manifold
with compact leaves. Then
\begin{enumerate}
\item The space of leaves $M/\calf_\xi$ is an almost K\"ahler
orbifold $\calz$ such that the canonical projection $\pi:
M\ra{1.3} M/\calf_\xi$ is an orbifold Riemannian submersion. \item
$M$ is the total space of a principal $S^1$ orbibundle (V-bundle)
over $M/\calf_\xi$ with connection $1$-form $\eta$ whose curvature
$d\eta$ is the pullback by $\pi$ of a symplectic form $\gro$ on
$M/\calf_\xi.$ \item The symplectic form $\gro$ defines a
non-trivial integral orbifold cohomology class, that is,
$[p^*\gro]\in H^2_{orb}(M/\calf_\xi,\bbz)$ where $p$ is the
natural projection of the orbifold classifying space (see Section 4.3 of \cite{BG05}).
\end{enumerate}
\end{theorem}
\begin{remark}\label{ratclass}
{\rm When $[p^*\gro]$ is an integral cohomology class in $H^2_{orb}(M/\calf_\xi,\bbz)$, the class $[\gro]$ is rational in the ordinary cohomology $H^2(M/\calf_\xi,\bbq).$ In fact, the orbifold K\"ahler classes are precisely the K\"ahler classes in $H^2(M/\calf_\xi,\bbq)$ that pullback under $p$ to integral classes in $H^2_{orb}(M/\calf_\xi,\bbz)$}
\end{remark}

Theorem \ref{bgthm} has an inversion (also known as orbifold Boothby-Wang) \cite{BG05}, namely

\begin{theorem}\label{kconinversionthm}
Let $(\calz,\gro,J)$ be an almost K\"ahler orbifold with
$[p^*\gro]\in H^2_{orb}(\calz,\bbz),$ and let $M$ denote the total
space of the circle V-bundle defined by the class $[\gro].$ Then
the orbifold $M$ admits a K-contact structure $(\xi,\eta,\Phi,g)$
such that $d\eta =\pi^*\gro$ where $\pi:M\ra{1.3} \calz$ is the
natural orbifold projection map. Furthermore, if all the local
uniformizing groups of $\calz$ inject into the structure group
$S^1$, then $M$ is a smooth K-contact manifold.
\end{theorem}

\subsection{Contact Isotopy}
Notice that in Theorem \ref{kconinversionthm} the choice of K-contact structure on $M$ is not uniquely determined by the symplectic structure $\gro$ on $\calz$. One chooses a connection 1-form $\eta$ such that $d\eta=\pi^*\gro,$ and this is determined only up to a gauge transformation $\eta\mapsto \eta +df$ for some $f\in C^\infty(M)$ that is invariant under $\xi$. Not even the contact structure $\cald$ is unique; what is unique, however, is the contact isotopy class. More generally, given contact metric structure $\cals=(\xi,\eta,\Phi,g)$ we can consider a deformation of the form $\eta\mapsto \eta_t=\eta+t\grz$ where $\grz$ is a basic 1-form with respect to the chararcteristic foliation $\calf_\xi$ that satisfies $\eta_t\wedge (d\eta_t)^n\neq 0$ for all $t\in[0,1]$. This clearly deforms the contact structure $\cald\mapsto \cald_t=\ker\eta_t$. However, Gray's Theorem \cite{Gra59,BG05} says that there are a family of diffeomorphisms (called a {\it contact isotopy}) $\varphi_t:M\ra{1.7} M$ and a family of positive smooth functions $f_t$ such that $\varphi_t^*\eta_t=f_t\eta$. So, although both the metric structures and almost CR structures have non-trivial deformations, the underlying contact structure does not, and there are no local invariants in contact geometry. If $\cals$ is K-contact (or Sasakian) with underlying contact structure $\cald$ then the entire family $$\cals^{\varphi_t}=((\varphi_t^{-1})_*\xi,f_t\eta,(\varphi_t^{-1})_*\circ \Phi_t\circ(\varphi_t)_*,\varphi_t^*g_t)$$
consists of K-contact (Sasakian) structures belonging to the contact structure $\cald$. So when we refer to a K-contact (or Sasakian) structure we often mean such a family defined up to contact isotopy.

\section{The Group of Contactomorphisms}\label{congroupsect}

\subsection{Basic Properties}
Let $\gD\gi\gf\gf(M)$ denote the group of diffeomorphisms of $M.$ We endow $\gD\gi\gf\gf(M)$ with the compact-open $C^\infty$ topology in which case it becomes a Fr\'echet Lie group. See \cite{Mil84,Ban97,Omo97,KrMi97} for the basics concerning infinite dimensional Lie groups and infinite dimensional manifolds. Since we deal almost exclusively with compact manifolds, the compact-open $C^\infty$ topology will suffice for our purposes. 

\begin{definition}\label{condiff}
Let $(M,\cald)$ be a closed connected contact manifold. Then we define the group $\gC\go\gn(M,\cald)$ of all {\bf contact diffeomorphisms} or {\bf contactomorphisms} by
$$\gC\go\gn(M,\cald)=\{\phi\in \gD\gi\gf\gf(M) ~|~ \phi_*\cald\subset \cald \}.$$
\end{definition}
The group $\gC\go\gn(M,\cald)$ is a regular infinite dimensional Fr\'echet Lie group \cite{Omo97} and an important invariant of a contact manifold. It is both locally contractible and locally path connected; hence, the identity component $\gC\go\gn_0(M,\cald)$ consists of contactomorphisms that are isotopic to the identity through contact isotopies. Recently it has been shown \cite{Ryb08} that the group $\gC\go\gn_0(M,\cald)$ is perfect\footnote{Recall that a group $G$ is perfect if $G$ equals its own commutator subgroup, that is, $G=[G,G]$. Since the commutator subgroup is normal, any non-Abelian simple group is perfect.}, in fact, it was shown independently to be simple \cite{Tsu08}. We denote by $\gC\go\gn(M,\cald)^+$ the subgroup of $\gC\go\gn(M,\cald)$ that preserves the orientation of $\cald,$ that is, it preserves the co-orientation of $(M,\cald).$ Clearly, $\gC\go\gn_0(M,\cald)\subset \gC\go\gn(M,\cald)^+.$

We can give the group $\gC\go\gn(M,\cald)$ an alternative formulation. If $\eta$ is a contact form, then it is easy to see that 
\begin{equation}
\gC\go\gn(M,\cald)=\{\phi\in \gD\gi\gf\gf(M)\, | \; \phi^*\eta=f^\eta_\phi\eta\, ,
 \;  f^\eta_\phi\in C^\infty(M)^*\}\, .
\end{equation}
Here, $C^\infty(M)^{*}$ denotes the subset of nowhere vanishing functions in 
$C^\infty(M)$, and $\gC\go\gn(M,\cald)^+$ is characterized by the subgroup such that $f_\phi>0.$  We also consider the subgroup $\gC\go\gn(M,\eta)$ of {\it strict contact transformations}, whose elements are those $\phi\in 
\gC\go\gn(M,\cald)$ such that $f_\phi=1$:
$$\gC\go\gn(M,\eta)=\{\phi\in \gD\gi\gf\gf(M)\, |\; \phi^*\eta=\eta \}\, .$$
This subgroup is also infinite dimensional; however, it is not a contact invariant.
It is well-known that the 1-parameter subgroup $\Xi$ generated by the flow of the Reeb vector field $\xi$ lies in the center of $\gC\go\gn(M,\eta).$ When $\cals=(\xi,\eta,\Phi,g)$ is a quasi-regular K-contact structure, the subgroup $\Xi$ is a closed normal subgroup of $\gC\go\gn(M,\eta)$ isomorphic to $S^1.$ The case that has been studied, mostly from the point of view of geometric quantization theory, is when $\eta$ is regular in which case the group $\gC\go\gn(M,\eta)$ has been called the {\it group of quantomorphisms} \cite{Sou70,RaSc81}.

\subsection{The Lie algebras}
The Lie algebras of these two groups are quite important. They are both Fr\'echet vector spaces which provide local model spaces for the corresponding Fr\'echet Lie groups. The first of these
is the Lie 
algebra of {\it infinitesimal contact transformations},
\begin{equation}\label{infcon}
\gc\go\gn(M,\cald)=\{X\in \X(M)\; | \; \pounds_X\eta=a(X)\eta\, , \; 
a(X)\in C^\infty(M)\}\, ,
\end{equation}
while the second is the subalgebra of {\it infinitesimal strict contact 
transformations}
\begin{equation}\label{strictcontacttrans}
\gc\go\gn(M,\eta)=\{X\in \X(M)\; | \; \pounds_X\eta =0\}\, .
\end{equation}

Let $\grG(\cald)$ denote the smooth sections of the contact bundle $\cald.$ Then
\begin{lemma}\label{Dcon}
$$\grG(\cald)\cap \gc\go\gn(M,\cald)=\{0\}.$$
\end{lemma}

\begin{proof}
Let $X\in \grG(\cald)\cap \gc\go\gn(M,\cald)$, then we have 
$$f_X\eta=\pounds_X\eta =d(\eta(X))+X\hook d\eta=X\hook d\eta$$
for some smooth function $f_X.$ Evaluating on $\xi$ gives $f_X=0$ since $\xi$ is the Reeb vector field of $\eta$ implying that $X\hook d\eta=0$ which in turn implies that $X=0$ by the non-degeneracy of $d\eta$ on $\cald.$
\end{proof}

Now a choice of contact form $\eta$ in $\gC^+(\cald)$ gives a well known Lie algebra isomorphism between $\gc\go\gn(M,\cald)$ and $C^\infty(M)$ given explicitly by
\begin{equation}\label{liealgiso}
 X\mapsto \eta(X)
\end{equation}
with the Lie algebra structure on $C^\infty(M)$ given by the Poisson-Jacobi bracket defined by $\{f,g\}=\eta([X_f,X_g]).$ 
Then one easily sees that the subalgebra $\gc\go\gn(M,\eta)$ is isomorphic to the subalgebra of $\Xi$-invariant functions $C^\infty(M)^\Xi.$ Moreover, we have
\begin{lemma}\label{closedalg}
$\gc\go\gn(M,\eta)$ is closed in $\gc\go\gn(M,\cald)$.
\end{lemma}

\begin{proof}
Using the isomorphism (\ref{liealgiso}) we have $\gc\go\gn(M,\eta)\approx C^\infty(M)^\Xi=\ker\xi$ as a differential operator. Any differential operator is continuous in the Fr\'echet topology. So $\ker\xi$ is closed, and the result follows by the isomorphism (\ref{liealgiso}).
\end{proof}

The following result is well known,

\begin{proposition}\label{centraleta}
Let $(M,\cald)$ be a contact manifold. For each choice of contact form $\eta$, the centralizer $\gz(\xi)$ of the Reeb vector field $\xi$ in  $\gc\go\gn(M,\cald)$ is precisely $\gc\go\gn(M,\eta)$.
\end{proposition}

\begin{proof}
This follows easily from the following lemma and the isomorphism (\ref{liealgiso}).
\end{proof}

We also have 

\begin{lemma}\label{R+comlem}
Let $(M,\cald)$ be a contact manifold. Then
\begin{enumerate}
\item $\calr^+(\cald)\subset \gc\go\gn(M,\cald),$
\item Let $\xi\in \calr^+(\cald)$ and $X\in \gc\go\gn(M,\cald)$. Then $[\xi,X]=0$ if and only if $\xi(\eta(X))=0.$
\end{enumerate}
\end{lemma}

\begin{proof}
We have
$$\eta([\xi,X])=-2d\eta(\xi,X)+\xi(\eta(X))-X(\eta(\xi))=\xi(\eta(X))$$
from which the only if part follows immediately. Conversely, $\xi(\eta(X))=0$ implies that $[\xi,X]$ is a section of $\cald.$ But also $[\xi,X]\in \gc\go\gn(M,\cald)$, so the result follows by Lemma \ref{Dcon}. 
\end{proof}

The identification between the Lie algebra $\gc\go\gn(M,\cald)$ and $C^\infty(M)$ can be thought of formally in terms of moment maps. For any contact form $\eta$ in $\gC^+(\cald)$, we define a map $\mu_\eta:M\ra{1.5} \gc\go\gn(M,\cald)^*$ (the algebraic dual) by 
\begin{equation}\label{momentmap}
\langle \mu_\eta(x),X\rangle=\eta(X)(x).
\end{equation}
Moreover, if $\eta'=f\eta$ is any other contact form in $\gC^+(\cald)$, then $\mu_{\eta'}=f\mu_\eta,$ and it is equivariant in the sense the for any $\phi\in \gC\go\gn(M,\cald)$, we have ${\rm Ad}^*_\phi\mu_\eta=\mu_{\phi^*\eta}=f_\phi\mu_\eta$ where $f_\phi$ is defined by $\phi^*\eta=f_\phi\eta.$

We also have

\begin{lemma}\label{R+lem}
Let $(M,\cald)$ be a compact contact manifold. Then
\begin{enumerate}
\item $\calr^+(\cald)$ is an open convex cone in $\gc\go\gn(M,\cald)$.
\item $\calr^+(\cald)$ is an invariant subspace under the adjoint action of $\gC\go\gn(M,\cald)$ on its Lie algebra $\gc\go\gn(M,\cald)$.
\item Every choice of $\eta\in \gC^+(\cald)$ gives an isomorphism of $\calr^+(\cald)$ with the cone of positive functions $C^\infty(M)^+$ in $C^\infty(M).$
\end{enumerate}
\end{lemma}

\begin{proof}
(1) was proven in \cite{Boy08}. To prove (2) we recall that the adjoint representation of any Lie group on its Lie algebra is defined by ${\rm Ad}_\phi(X)=(L_\phi\circ R_{\phi^{-1}})_*X =(R_{\phi^{-1}})_*X$. Thus, for $\xi'\in \calr^+(\cald)$ we have
$$\eta({\rm Ad}_\phi(\xi'))=f_{\phi^{-1}}\eta(\xi')\circ\phi^{-1}>0.$$
(3) is clear from the definition.
\end{proof}

\subsection{Fr\'echet Lie Groups}
The next result shows that $\gC\go\gn(M,\eta)$ is a closed Fr\'echet Lie subgroup of $\gC\go\gn(M,\cald).$ In particular, it implies that $\gC\go\gn(M,\eta)$ is a Fr\'echet Lie group for any contact form $\eta$, a general proof of which I have been unable to find in the literature (when $\eta$ is regular it is Theorem 7.3 in \cite{Omo97}). 

\begin{proposition}\label{etasub}
Let $(M,\cald)$ be a compact contact manifold. Then for any contact form $\eta$ the subgroup $\gC\go\gn(M,\eta)$ is a closed submanifold of $\gC\go\gn(M,\cald).$ Hence, $\gC\go\gn(M,\eta)$ is a Fr\'echet Lie subgroup of $\gC\go\gn(M,\cald).$
\end{proposition} 

\begin{proof}
We need to show that $\gC\go\gn(M,\eta)$ is a Fr\'echet submanifold of $\gC\go\gn(M,\cald).$ Since contact structures are first order and we are working in the $C^\infty$-Fr\'echet category, checking $C^1$-closeness is enough to imply $C^\infty$-closeness. The proof uses the construction of local charts due to Ly\v{c}hagin \cite{Lyc75,Lyc79} as presented in \cite{Ban97}. This construction is an adaptation of an argument of Weinstein in the symplectic category. Choosing a 1-form $\eta$, one identifies an element $\phi\in \gC\go\gn(M,\cald)^+$ with its graph $\grG_\phi$ viewed as a Legendrian submanifold of the contact manifold $M\times M\times \bbr^+$ with contact form $\gra=r\pi_1^*\eta-\pi_2^*\eta,$ where $\pi_i$ is the projection onto the first (second) factor for $i=1(2),$ respectively, and $r$ is the coordinate of $\bbr^+.$ Then $(M,\eta)$ is a Legendrian submanifold of the contact manifold $(M\times M\times \bbr,\gra)$ precisely when the graph of a diffeomorphism is the graph of a contactomorphism, i.e. given by $\grG_\phi(x)=(x,\phi(x),f_\phi(x)).$ Notice that the graph of the identity map is $\grG_{\rm id}=\grD\times \{1\}$, where $\grD$ is the diagonal in $M\times M.$ Locally one then models this in terms of Legendrian submanifolds of the 1-jet bundle $J^1(M).$ Here submanifolds are Legendrian with respect to the canonical 1-form $\theta$ on $J^1(M)$ precisely when the sections of $J^1(M)$ are {\it holonomic}, that is, the 1-jets of a function. So there is a contactomorphism $\Psi$ between a tubular neighborhood $\caln$ of $\grD\times \{1\}$ in $M\times M\times \bbr^+$ and a tubular neighborhood $\calv$ of the zero section in $J^1(M)$ which satisfies $\Psi^*\theta=\gra=r\pi_1^*\eta-\pi_2^*\eta.$ So if $U$ is a neighborhood of the identity in $\gC\go\gn(M,\cald),$ there is a diffeomorphism $\varTheta$ between $U$ and the subspace of $\caln$ consisting of Legendrian submanifolds in $\caln$  defined by $\varTheta(\phi)=\grG_\phi$. Composing this with $\Psi$ gives a diffeomorphism $\Psi\circ\varTheta$ from $U$ to the subspace of holonomic sections of $J^1(M)$ in $\calv$ given by 
\begin{equation}\label{Lycmap}
\Psi\circ\varTheta(\phi)=j^1(F_\phi)
\end{equation} 
where the function $F_\phi$ is uniquely determined by $\phi.$

As a Fr\'echet space $\gC\go\gn(M,\cald)$ is modelled on the Fr\'echet vector space $\gc\go\gn(M,\cald)$, and $\gC\go\gn(M,\eta)$ is modelled on the Fr\'echet vector space $\gc\go\gn(M,\eta)$ which is closed in $\gc\go\gn(M,\cald)$ by Lemma \ref{closedalg}. However, by the isomorphism (\ref{liealgiso}) it is convenient to
use the vector space $C^\infty(M)$ as the local model. The reason being that we can identify $C^\infty(M)$ with the holonomic sections of $J^1(M),$ that is, we view $C^\infty(M)$ as the subspace of holonomic sections $C^\infty(M)\subset \grG(M,J^1(M))$ with the identification $F\leftrightarrow j^1(F).$ This gives a diffeomorphism $\Psi\circ\varTheta$ between $U$ and a neighborhood of zero in $C^\infty(M)$ which provide the local charts for $\gC\go\gn(M,\cald)$.

Now let $\phi\in \gC\go\gn(M,\eta)$, then $F_\phi\in C^\infty(M)^\Xi.$ Consider $tF_\phi\in C^\infty(M)^\Xi$ for all $t\in [0,1],$ and let $\phi_t$ denote the contactomorphism corresponding to $j^1(tF_\phi)$ under the diffeomorphism $(\Psi\circ\Theta)^{-1}.$ Then $\phi_0={\rm id}$ and $\phi_1=\phi.$ Since $tF_\phi$ is $C^1$-close to $0$ for all $t\in [0,1]$ and $j^1(F_\phi)$ is $C^2$-close to $0$, we have 
$$||Dj^1(tF_\phi)||\leq ||Dj^1(F_\phi)||< \gre, \qquad \forall t\in [0,1]$$ 
whenever $\phi$ is $C^1$-close to the identity in $\gC\go\gn(M,\cald)$. So $tj^1(F_\phi)\in V,$ implying that $\phi_t=(\Psi\circ\Theta)^{-1}(tj^1(F_\phi))$ is in $U.$ Now since 
$$\Psi(\grG_{\phi_t})=j^1(tF_\phi)\in C^\infty(M)^\Xi,$$
$\grG_{\phi_t}$ must have the form $\grG_{\phi_t}(x)=(x,\phi_t(x),1)$ for all $t\in [0,1]$ which implies that $\phi_t\in  U\cap \gC\go\gn(M,\eta)$ for all $t\in [0,1].$ This shows that $\Psi\circ\Theta( U\cap \gC\go\gn(M,\eta))= \Psi\circ\Theta(U)\cap \gc\go\gn(M,\eta)$ which implies that $\gC\go\gn(M,\eta)$ is an embedded submanifold of $\gC\go\gn(M,\cald)$.
\end{proof}

\subsection{Tori in $\gC\go\gn(M,\cald)$}
Next we mention an important observation of Lerman \cite{Ler02a,BG05} concerning Lie group actions that preserve the contact structure. Here I state this only for the case at hand, namely, compact Lie groups. 

\begin{lemma}\label{lerlem}
Let $(M,\cald)$ be a compact contact manifold, and 
let $G$ be a compact Lie subgroup of $\gC\go\gn_0(M,\cald).$ Then there exists a $G$-invariant contact form $\eta\in \gC^+(\cald)$ such that $\cald ={\rm ker}{~\eta},$ or otherwise stated: there exists a contact 1-form $\eta$ in the contact structure such that $G\subset \gC\go\gn(M,\eta).$
\end{lemma}

In particular, we are interested in subtori in $\gC\go\gn_0(M,\cald).$ Such tori are partially ordered by inclusion, and we are particularly interested in maximal tori in $\gC\go\gn_0(M,\cald).$ Recall the following

\begin{definition}\label{maxtori}
A torus $T\subset G$ is said to be {\bf maximal} in a group $G$ if $T'\subset G$ is any other torus with $T\subset T'$, then $T'=T.$
\end{definition}

We are interested in maximal tori of the contactomorphism group $\gC\go\gn(M,\cald)$ of a contact manifold $M.$ One should not confuse the definition of maximal given here with that of a contact manifold being toric. For such a manifold a maximal torus has maximal dimension, namely $n+1$ when $\dim M=2n+1.$ In principal there can be many maximal tori, and even having different dimensions. For any torus $T$ we denote its dimension by $\gr(T)$. Then for any torus $T\subset \gC\go\gn(M,\cald)$, we have $1\leq \gr(T)\leq n+1.$ Generally, it may even be true that $\gC\go\gn(M,\cald)$ contains no torus. But this does not happen for contact structures of K-contact type on a compact manifold. We are mainly concerned with special types of torus actions introduced earlier \cite{BG00b,BG05}. 

\begin{definition}\label{Reebtype} Let $(M,\cald)$ be a contact manifold. We say that a torus $T\subset \gC\go\gn(M,\cald)$ is of {\bf Reeb type} if there is a contact 1-form $\eta$ in $\cald$ such that its Reeb vector field lies in the Lie algebra $\gt$ of $T.$ 
\end{definition}

If $T$ is a torus of Reeb type then the set of Reeb vector fields with the co-orientation of $\cald$ fixed that lie in $\gt$ is the open convex cone $\gt^+=\gt\cap \calr^+(\cald)$ in $\gt.$ The following which follows directly from Proposition \ref{centraleta} was noticed by Lerman \cite{Ler02b}:

\begin{lemma}\label{lerlem2}
If $T\subset \gC\go\gn(M,\cald)$ is a torus of Reeb type, then $T\subset \gC\go\gn(M,\eta)$ where $\eta$ is a contact form of $\cald$ whose Reeb vector field lies in the Lie algebra $\gt$ of $T$.
\end{lemma}

As we shall see below there can be many such $\eta.$ Moreover, they will all give K-contact structures, that is:

\begin{proposition}\label{reebkcon}
A compact contact manifold $(M,\cald)$ has a torus of Reeb type in its contactmorphism group $\gC\go\gn(M,\cald)$ if and only if it has a compatible K-contact metric structure. Hence, $\cald$ is of K-contact type.
\end{proposition}

\begin{proof}
Suppose $(M,\cald)$ is of K-contact type, then the closure of the any Reeb orbit is a torus, which by definition is of Reeb type. Conversely, if $\gC\go\gn(M,\cald)$ has a torus of Reeb type, Lemma \ref{lerlem2} implies that $T\subset \gC\go\gn(M,\eta)$ for some contact 1-form $\eta.$ Let $g$ be a Riemannian metric compatible with $\eta.$ We can integrate over $T$ to obtain a $T$-invariant metric which we also denote by $g$, but generally we lose compatibility. However, since $T$ contains the flow of the Reeb vector field $\xi$, the characteristic foliation $\calf_\xi$ is Riemannian, and as in Theorem 1 of \cite{Ruk95a} we obtain a $\xi$-invariant compatible metric $g$. But since $d\eta,\xi$ and $g$ determine $\Phi$ and all are invariant under $\xi$, the newly constructed contact metric structure $\cals=(\xi,\eta,\Phi,g)$ will be K-contact. 
\end{proof}

Let $T\subset \gC\go\gn(M,\cald)$ be a torus of Reeb type and consider the moment map (\ref{momentmap}) restricted to $T$, that is, $\mu_\eta:M\ra{1.5} \gt^*$ given by 
\begin{equation}\label{momentmap2}
\langle \mu_\eta(x),\grt\rangle=\eta(\grt)(x)
\end{equation}
where $\eta$ is any 1-form in $\gC(\cald)$ whose Reeb vector field belongs to the Lie algebra $\gt$ of $T,$ and $\gt^*$ is the dual co-algebra. We have

\begin{lemma}\label{not0}
Let $T\subset \gC\go\gn(M,\cald)$ be a torus of Reeb type. Then $0\in \gt^*$ is not in the image of the moment map $\mu_{\eta'}$ for any $\eta'\in \gC(\cald).$ 
\end{lemma}

\begin{proof}
Since $T$ is of Reeb type there is an $\eta\in \gC(\cald)$ whose Reeb vector field $\xi$ lies in $\gt$, so the image $\mu_\eta(M)$ lies in the characteristic hyperplane $\langle\mu_\eta(x),\xi\rangle=1$. Then for any other $\eta'\in \gC(\cald)$ we have $\eta'=\eta'(\xi)\eta$ with $\eta'(\xi)(x)\neq 0$ for all $x\in M.$ So we have $\mu_{\eta'}(x)=\eta'(\xi)(x)\mu_\eta(x)$ for all $x\in M$ proving the result. 
\end{proof}

\begin{remark}\label{notReeb}
{\rm It is clear that the lemma fails generally for tori that are not of Reeb type. Just take any circle such that $\eta(\grt)=0$ has solutions. However, it does hold for any toric contact structure whether or not it is of Reeb type \cite{Ler02a}}.
\end{remark}

In order to treat the contact 1-forms on an equal footing, we consider the annihilator $\cald^o$ of $\cald$ in $T^*M$. Choosing a contact 1-form $\eta$ trivializes the bundle
$\cald^o\approx M\times \bbr,$ as well as chooses an orientation
of $\cald^o,$ and splits $\cald^o\setminus \{0\}=\cald^o_+\cup \cald^o_-$. Then as in \cite{Ler02a} we can think of the moment map (\ref{momentmap2}) as a map $\Upsilon:\cald^o_+\ra{1.6} \gt^*$ given by
\begin{equation}\label{tormom}
\langle\Upsilon(x,\eta),\grt\rangle =\langle \mu_\eta(x),\grt\rangle=\eta(\grt)(x).
\end{equation}
The image of this moment map is a convex cone in $\gt^*$ without its cone point. Then the {\it moment cone} $C(\cald,T)$ is ${\rm im}\Upsilon\cup \{0\}$.

Lerman \cite{Ler02a} also observed that in the case of toric contact structures, a choice of norm $||\cdot||$ in $\gt^*$ picks out a unique 1-form $\eta\in\gC(\cald)$ such that $||\mu_\eta(x)||=1.$ This works equally well for tori of Reeb type. Let $\xi\in\gt^+$ and for any $\eta'\in \gC^+(\cald)$ set $||\mu_{\eta'}||=\eta'(\xi)$, then the contact form $\eta$ whose Reeb vector field is $\xi$ satisfies $||\mu_{\eta}||=1,$ and any other $\eta'\in \gC^+(\cald)$ is given by $\eta'=||\mu_{\eta'}||\eta.$

We are interested in the set of conjugacy classes of maximal tori in $\gC\go\gn(M,\cald)$. Note the following:
\begin{lemma}\label{adReeb}
If $T\subset \gC\go\gn(M,\cald)$ is a maximal torus of Reeb type then so is ${\rm Ad}_\phi T$ for all $\phi\in \gC\go\gn(M,\cald)$.
\end{lemma}
 
\begin{proof}
One easily checks that if $\xi$ is the Reeb vector field of $\eta$ in $\gt,$ then ${\rm Ad}_\phi\xi$ is the Reeb vector field of $\phi^*\eta$ in ${\rm Ad}_\phi\gt.$ 
\end{proof}

\begin{definition}\label{conclass}
Let $\gS\gC_T(\cald)$ denote the union of the empty set $\emptyset$ with the set of conjugacy classes of tori in $\gC\go\gn(M,\cald)$, and let $\gS\gC_T^{max}(\cald)$ denote the subset of conjugacy classes of maximal tori. We let $\gn(\cald)$ denote the cardinality of $\gS\gC_T^{max}(\cald)$, and $\gn(\cald,\gr)$ denote the cardinality of the subset $\gS\gC_T^{max}(\cald,\gr)$ of $\gS\gC_T^{max}(\cald)$ consisting of conjugacy classes of maximal tori in $\gC\go\gn(M,\cald)$ of dimension $\gr$. Similarly, we let $\gS\gC_{RT}(\cald), \gS\gC_{RT}^{max}(\cald),$ and $\gS\gC_{RT}^{max}(\cald,\gr)$ denote the corresponding subsets of $\gS\gC_T(\cald)$ ( consisting of conjugacy classes of tori of Reeb type. Then we denote by $\gn_R(\cald)$ (respectively, $\gn_R(\cald,\gr)$) the cardinality of $\gS\gC_{RT}^{max}(\cald)$ (respectively, of $\gS\gC_{RT}^{max}(\cald,\gr)$). 
\end{definition}

Note that $\gS\gC_T(\cald)$ is partially ordered by inclusion, and $0\leq \gr\leq n+1.$
Clearly we have $\gn_R(\cald,\gr)\leq \gn(\cald,\gr)$ and $\gn_R(\cald)\leq \gn(\cald),$ and all are invariants of the contact structure. So we will sometimes write $\gn([\cald])$ and $\gn_R([\cald]),$ etc., where $[\cald]$ denotes the isomorphism class of the contact structure $\cald.$ Note also that we have decompositions 
\begin{equation}\label{decomp}
\gS\gC_T^{max}(\cald)=\bigsqcup_{\gr=0}^{n+1}\gS\gC_T^{max}(\cald,\gr),\qquad \gS\gC_{RT}^{max}(\cald)=\bigsqcup_{\gr=1}^{n+1}\gS\gC_{RT}^{max}(\cald,\gr)
\end{equation}
where we set $\gS\gC_T^{max}(\cald,0)=\emptyset.$ We have also $\gS\gC_{RT}^{max}(\cald,\gr)\subset \gS\gC_T^{max}(\cald,\gr)$ and $\gS\gC_{RT}^{max}(\cald)\subset \gS\gC_T^{max}(\cald)$.

\section{Compatible Almost Complex structures}\label{secalmostcomp}

Recall from Section \ref{almostCR} that the space of compatible almost complex structures $\calj(\cald)$ on the contact bundle $\cald$ can be identified with partially integrable strictly pseudoconvex almost CR structures. So we think of $\calj(\cald)$ as the {\it space of all partially integrable strictly pseudoconvex almost CR structures $(\cald,J)$ whose codimension one subbundle of $TM$ is $\cald.$} Alternatively, $\calj(\cald)$ can be thought of as the space of all splittings $\cald\otimes \bbc=\cald^{1,0}\oplus \cald^{0,1}$ where $\cald^{1,0}(\cald^{0,1})$ are the $\pm i$ eigenspaces of $J,$ respectively. The set $\calj(\cald)$ is given the topology as a subspace of $\grG(M,{\rm End}(\cald))$, and with this topology we have

\begin{proposition}\label{Jcontr}
Let $(M,\cald)$ be a compact contact manifold. Then
\begin{enumerate}
\item The space $\calj(\cald)$ is non-empty and contractible.
\item $\calj(\cald)$ is a smooth Fr\'echet manifold whose tangent space $T_J\calj(\cald)$ at $J\in \calj(\cald)$ is the Fr\'echet vector space of smooth symmetric (with respect to $g_\cald$) endomorphisms of $\cald$ that anticommute with $J$. 
\end{enumerate} 
\end{proposition}

\begin{proof} 
The proof of (1) is standard and can be found, for example, in \cite{Gei08}.
The proof of (2) is straightforward and can be found in \cite{Smo01}.
\end{proof}

%Following \cite{AGK09} we briefly describe the K\"ahler structure on $\calj(\cald)$.

\subsection{The Subgroup of Almost CR Transformations}\label{almostCRtranssec}
Let $(\cald,J)$ be an almost CR structure on $M.$ We define the {\it group of almost CR transformations} by 
\begin{equation}\label{CRtrans}
\gC\gR(\cald,J)=\{\phi\in \gD\gi\gf\gf(M)\; | \; \phi_*\cald\subset \cald,~  \phi_*J=J\phi_*\}\, .
\end{equation}
If $(\cald,J)$ is strictly pseudoconvex $\gC\gR(\cald,J)$ is known \cite{ChMo74,BRWZ04} to be a Lie transformation group.
Moreover, in this case $\gC\gR(\cald,J)$ preserves the co-orientation, so $\gC\gR(\cald,J)\subset\gC\go\gn(M,\cald)^+$. Suppose now that $(\cald,J)$ is a contact structure with a compatible almost complex structure. Then
the group $\gC\go\gn(M,\cald)^+$ acts smoothly on the manifold $\calj(\cald)$ by $J\mapsto \phi_*J\phi^{-1}_*$ for $\phi\in \gC\go\gn(M,\cald)^+$ and partitions it into orbits. We say that $J$ and $J'$ in $\calj(\cald)$ are {\it equivalent} if they lie on the same orbit under this action. The isotropy subgroup at $J\in \calj(\cald)$ is precisely the group $\gC\gR(\cald,J),$ and it follows from Theorem III.2.2 of \cite{Omo97} that $\gC\gR(\cald,J)$ is a closed Lie subgroup of $ \gC\go\gn(M,\cald)^+.$
The Lie algebra $\gc\gr(\cald,J)$ of  $\gC\gR(\cald,J)$
can be characterized as
\begin{equation}\label{infcrtrans}
\gc\gr(\cald,J)=\{X\in \gc\go\gn(M,\cald)\; | \;  \pounds_XJ=0\}\, .
\end{equation}

When the strictly pseudoconvex almost CR structure is integrable, the fact that $\gC\gR(\cald,J)$ is a compact Lie group except for the standard CR structure on the sphere $S^{2n+1}$ has a long and varied history \cite{ChMo74,Web77,Sch95,Lee96}. However, recent results \cite{CaSc00,Fra07} in the study of Cartan geometries have now allowed one to weaken the hypothesis to that of partial integrability.

\begin{theorem}\label{Schthm} Let $(\cald,J)$ be a contact structure with a compatible almost complex structure on a compact manifold $M$. Then $\gC\gR(\cald,J)$ is a compact Lie group except when $(\cald,J)$ is the standard CR structure on $S^{2n+1}$ in which case $\gC\gR(\cald,J)=SU(n+1,1).$
\end{theorem}

\begin{proof}
By Proposition \ref{CRcon} $(\cald,J)$ is a partially integrable strictly pseudoconvex almost CR structure. But according to  \cite{CaSc00} (see also \cite{Arm08}) a partially integrable strictly pseudoconvex almost CR structure is a parabolic geometry with a canonical Cartan connection. So the result follows from Theorem 1 of \cite{Fra07} 
\end{proof}

So when $M$ is compact $\gC\gR(\cald,J)$ will always be a compact Lie group except when $M$ is the standard CR structure on $S^{2n+1}$ which is well understood (cf. \cite{BGS06} and references therein). Since $\gC\gR(\cald,J)=SU(n+1,1)$ in the case of the standard sphere, it has a unique maximal torus up to conjugacy. 

\subsection{Conjugacy Classes of Maximal Tori and Compatible Almost Complex Structures}
Theorem \ref{Schthm} allows us to define a map (not necessarily continuous) as follows: given a compatible almost complex structure $J$ we associate to it the unique conjugacy class $\calc_T(\cald)$ of maximal torus in $\gC\gR(\cald,J)$ if $\dim\gC\gR(\cald,J)>0$, and the empty conjugacy class if $\dim\gC\gR(\cald,J)=0.$ This in turn gives a conjugacy class $\overline{\calc_T(\cald)}$ of tori in the contactomorphism group $\gC\go\gn(M,\cald)$. However, the tori may not be maximal in $\gC\go\gn(M,\cald)$, nor lie in a unique maximal torus even if the almost complex structure is integrable. Nevertheless, there is a sufficiently interesting subspace of $\calj(\cald)$ where the conjugacy class $\overline{\calc_T(\cald)}$ is maximal in $\gC\go\gn(M,\cald)$. For example, by Theorem 5.2 of \cite{BG00b} any toric contact manifold of Reeb type has compatible complex structures $J$ and a maximal torus in $\gC\gR(\cald,J)$ is maximal in $\gC\go\gn(M,\cald)$. There are other non-toric examples treated in \cite{BoTo11}. 

\begin{example}\label{counter}[counterexample] Here are examples of transverse complex structures $J_\tau$ on $T^2\times S^3$ taken from \cite{BoTo11} where the maximal torus in $\gC\gR(\cald_{k,1},J_\tau)$ is not maximal in $\gC\go\gn(T^2\times S^3,\cald_{k,1})$ and, in fact, lies in several non-conjugate maximal tori in $\gC\go\gn(T^2\times S^3,\cald_{k,1})$. Here $\tau$ is a representative complex structure in the moduli space of complex structures on $T^2$. Let $A_{0,\tau}$ denote the non-split complex structure on $T^2\times S^2$ associated with $\tau$. This is compatible with the symplectic form $\gro_{k,1}$ described in Section 4.1 of \cite{BoTo11}. As shown in Section 7 of \cite{BoTo11} this lifts to a Sasakian structure with transverse complex structure $J_\tau$ on $T^2\times S^3$. Moreover, as shown there the only Hamiltonian holomorphic vector field of $J_\tau$ is the Reeb vector field $\xi_{k,1}$. Thus, a maximal torus of $\gC\gR(\cald_{k,1},J_\tau)$ has dimension one and is generated by the Reeb field.  But it follows from Proposition 7.2 of \cite{BoTo11} that this maximal torus is contained in $k$ non-conjugate tori of $\gC\go\gn(T^2\times S^3,\cald_{k,1}).$
\end{example}

Now generally we define a map
\begin{equation}\label{Jmap}
\gQ:\calj(\cald)\ra{1.6} \gS\gC_T(\cald) ~\text{by} ~\gQ(J)=\overline{\calc_T(\cald)},
\end{equation}
that is, $\gQ(J)$ is the conjugacy class of tori in $\gC\go\gn(M,\cald)$ defined by the unique maximal conjugacy class of tori in $\gC\gR(\cald,J)$. Since for all $\phi\in\gC\go\gn(M,\cald)$
\begin{equation}\label{conjtorcont}
\gQ(\phi_*J\phi^{-1}_*)=\overline{\phi \calc_T(\cald)\phi^{-1}}=\overline{\calc_T(\cald)},
\end{equation}
this decends to a map from the $\gC\go\gn(M,\cald)$-orbits in $\calj(\cald)$ to the set $\gS\gC_T(\cald)$ of conjugacy classes of maximal tori in $\gC\go\gn(M,\cald)$, namely 
\begin{equation}\label{Jmap2}
 \bar{\gQ}:\calj(\cald)/\gC\go\gn(M,\cald)\ra{1.6} \gS\gC_T(\cald) ~\text{given by}~\bar{\gQ}([J])=\overline{\calc_T(\cald)}.
\end{equation}
We are mainly interested in those almost complex structures $J$ for which $\gQ(J)\in \gS\gC_T^{max}(\cald)$, that is, for which the conjugacy class $\overline{\calc_T(\cald)}$ consists of maximal tori in $\gC\go\gn(M,\cald)$.

As with conjugacy classes there is a grading of $\calj(\cald)$ according to the rank of $\gC\gR(\cald,J).$ We let $\calj(\cald,\gr)$ denote the compatible almost complex structures such that $\gC\gR(\cald,J)$ has rank $\gr,$ with $\gr=0$ understood if $\gC\gR(\cald,J)$ is finite. Then we have 
\begin{equation}\label{decompJ}
\calj(\cald)=\bigsqcup_{\gr=0}^{n+1}\calj(\cald,\gr).
\end{equation}
The map $\gQ$ clearly preserves the respective decompositions, that is, for each $\gr=0,\ldots,n+1$, $\gQ$ restricts to a map
\begin{equation}\label{Jmap3}
\gQ_\gr:\calj(\cald,\gr)\ra{1.6} \gS\gC_T(\cald,\gr). 
\end{equation}
Clearly the action of $\gC\go\gn(M,\cald)$ also preserves the grading so we have a map on the quotient
\begin{equation}\label{Jmap4}
 \bar{\gQ}_\gr:\calj(\cald,\gr)/\gC\go\gn(M,\cald)\ra{1.6} \gS\gC_T(\cald,\gr).
\end{equation}

By a standard result (cf. Lemma A.4 of \cite{Kar99}) we have
\begin{theorem}\label{surjthm}
The maps (\ref{Jmap}),(\ref{Jmap2}),(\ref{Jmap3}) and (\ref{Jmap4}) are surjective. 
\end{theorem}

However, it is easy to see that except for the toric case the maps (\ref{Jmap2}) and (\ref{Jmap4}) typically are not injective. This is studied further in Sections \ref{torrev} and \ref{whs}.

\subsection{The Group $\gA\gu\gt(\cals)$}
Let $\cals=(\xi,\eta,\Phi,g)$ be a contact metric structure. The group $\gA\gu\gt(\cals)$ is the group of automorphisms of $\cals$ defined by
\begin{equation}\label{autgroup}
\gA\gu\gt(\cals)=\gA\gu\gt(\xi,\eta,\Phi,g)= \{\varphi\in \gC\go\gn(M,\cald)~|~ (\varphi^{-1}_*\xi,\varphi^*\eta,\varphi_*\Phi\varphi^{-1}_*,\varphi^*g)=(\xi,\eta,\Phi,g)\}.
\end{equation}
We can abbreviate this condition as $\cals^\varphi=\cals.$ Note that $\gA\gu\gt(\cals)$ is a compact Lie subgroup of $\gC\gR(\cald,J)$. The following result was given in \cite{BGS06} for the case that $\cals$ is of Sasaki type, but, in the light of Theorem \ref{Schthm} the proof given there works equally as well under the weaker assumption that $J$ is partially integrable.

\begin{theorem}\label{autcr}
Let $\cald$ be a contact structure with a compatible almost complex structure $J$ on a compact manifold
$M$. Then there exists a contact metric structure $\cals=(\xi,\eta,\Phi,g)$ with underlying almost 
{\rm CR} structure $(\cald,J)$, whose automorphism group $\gA\gu\gt
(\cals)$ is a maximal compact subgroup of $\gC\gR(\cald,J)$. In fact,
except for the case when $(\cald,J)$ is the standard {\rm CR} structure on the sphere 
$S^{2n+1}$, the automorphism 
group $\gA\gu\gt(\cals)$ of $\cals$ is equal to
$\gC\gR(\cald,J)$. Moreover, if $(\cald,J)$ is of K-contact or Sasaki type, we can take $\cals$ to be K-contact or Sasakian, respectively.
\end{theorem}

\begin{proof}
By Theorem \ref{Schthm} $\gC\gR(\cald,J)$ is compact except when $(\cald,J)$ is the standard CR structure on the sphere in which case the result holds by \cite{BGS06}. So when $(\cald,J)$ is not the standard CR structure on the sphere, $\gC\gR(\cald,J)$ is compact, and by Lemma \ref{lerlem} we can assume that $\gC\gR(\cald,J)\subset \gC\go\gn(M,\eta)$ for some $\eta.$ But then from Proposition 8.1.1 of \cite{BG05} we see that there is a contact metric structure $\cals=(\xi,\eta,\Phi,g)$ such that $\gA\gu\gt(\cals)=\gC\go\gn(M,\eta)\cap\gC\gR(\cald,J)= \gC\gR(\cald,J)$. To prove the last statement we notice that if $\cals$ is of K-contact type, there is a Reeb vector field of some $\eta\in \gC^+(\cald)$ in $\gc\gr(\cald,J)$, and the closure of this Reeb field is a torus of Reeb type. So the result follows from Lemma \ref{lerlem2}. 
\end{proof}

\subsection{Sasaki Cones and the Sasaki Bouquet}\label{secSascone}
The Sasaki cone was defined in \cite{BGS06} (Definition 6.7)\footnote{In \cite{BGS06} the Sasaki cone was defined for Sasakian structures, but the notion clearly holds also for K-contact structures.}  to be the moduli space of Sasakian structures whose underlying strictly pseudoconvex CR structure is $(\cald,J).$ The construction goes as follows: fix an almost CR structure of K-contact type, $(\cald,J)$, and let $\eta$ be a contact form representing $\cald$ (with a fixed co-orientation). Define $\gc\gr^+(\cald,J)$ to be the subset of the Lie algebra $\gc\gr(\cald,J)$ defined by
\begin{equation}\label{cr+}
\gc\gr^+(\cald,J)=\{\xi'\in \gc\gr(\cald,J)~|~\eta(\xi')>0\}.
\end{equation}
The set $\gc\gr^+(\cald,J)$ is a convex cone in $\gc\gr(\cald,J)$, and it is open when $M$ is compact. Furthermore, $\gc\gr^+(\cald,J)$ is invariant under the adjoint action of $\gC\gR(\cald,J)$, and is bijective to the set of K-contact structures compatible with $(\cald,J),$ namely

\begin{equation}
\calk(\cald,J)=\left\{
\begin{array}{c}
 \cals=(\xi,\eta,\Phi,
g):\; \cals \; {\rm a~ \text{K-contact}~ structure} \\
({\rm ker}\, \eta, \Phi \mid_{{\rm ker}\, \eta})=
(\cald,J)\end{array} \right\}.
\end{equation}
(Recall that a contact metric structure $\cals=(\xi,\eta,\Phi,g)$ with underlying almost CR structure $(\cald,J)$ is K-contact if and only if the Reeb vector field $\xi\in\gc\gr(\cald,J).$) We have

\begin{definition}\label{Sasakicone}
Let $(\cald,J)$ be a strictly pseudoconvex almost CR structure of K-contact type on a compact manifold $M$. The {\bf Sasaki cone} is defined to be the moduli space of K-contact structures with underlying CR structure $(\cald,J)$, that is
$$\grk(\cald,J)=\calk(\cald,J)/\gC\gR(\cald,J)=\gc\gr^+(\cald,J)/\gC\gR(\cald,J).$$
\end{definition}

So the Sasaki cone is an invariant of the underlying strictly pseudoconvex almost CR structure. Although different in nature, it plays a role similar to the K\"ahler cone in K\"ahler and almost K\"ahler geometry. If the transverse almost complex structure $J$ is integrable, then the Sasaki cone $\grk(\cald,J)$ is the moduli space of all Sasakian structures with underlying CR structure $(\cald,J)$. This is the case treated in \cite{BGS06}. It is convenient to give a description of the Sasaki cone in terms of maximal tori. Let $G$ be a maximal compact subgroup of $\gC\gR(\cald,J)$, and fix a maximal torus $T\subset G.$ Let $\calw(\cald,J)$ denote the Weyl group of $G$, $\gt$ the Lie algebra of $T$ and define $\gt^+(\cald,J)=\gt\cap \gc\gr^+(\cald,J).$ By abuse of terminology I often refer to $\gt^+(\cald,J)$ as the Sasaki cone, or perhaps more precisely the {\it unreduced Sasaki cone}. Note that $\gt^+(\cald,J)$ is the dual cone to the interior of the moment cone.

Choosing a representative torus in its conjugacy class, we have the identification 
\begin{equation}\label{Sasakicone2}
\grk(\cald,J)= \gt^+(\cald,J)/\calw(\cald,J).
\end{equation}
When the CR structure is fixed the Reeb vector field uniquely determines the K-contact (Sasakian) structure. However, somewhat more holds.

\begin{lemma}\label{Tequivlem}
Let $(\cald,J)$ and $(\cald,J')$ be CR structures of K-contact type, and suppose that $\gQ(J)=\gQ(J')$. Then $\grk(\cald,J')\approx \grk(\cald,J)$, where the equivalence $\approx$ is induced by conjugation.
\end{lemma}

\begin{proof}
The Lie algebras $\gt$ and $\gt'$ of the maximal tori $T$ and $T'$ are conjugate under an element of $\gC\go\gn(M,\cald)$, and one easily sees that this implies that their Weyl groups $\calw$ and $\calw'$ are also conjugate. Furthermore, the positivity condition only depends on the conjugacy class. So the result follows by using the identification (\ref{Sasakicone2}).
\end{proof}

It is convenient to give
\begin{definition}\label{Tequivdef}
We say that $J,J'\in \calj(\cald)$ are {\bf $T$-equivalent} if $\gQ(J)=\gQ(J')$.
\end{definition}

Lemma \ref{Tequivlem} allows us to define a certain configuration of K-contact structures associated to a given contact structure of K-contact type.

\begin{definition}\label{Sasbou}
Let $\calj_T(\cald)$ denote the set of all $T$-equivalence classes of almost complex structures in $\calj(\cald)$. We define {\bf the (complete) bouquet of Sasaki cones} $\gB(\cald)$ as 
$$\bigcup_\gra\grk(\cald,J_\gra)$$ 
where the union is taken over one representative $J_\gra$ in each $T$-equivalence class in $\calj_T(\cald)$. If $\cala\subset  \calj_T(\cald)$, we say that 
$$\gB_\cala(\cald)=\bigcup_{\gra\in \cala}\grk(\cald,J_\gra)$$
is {\bf a bouquet of Sasaki cones} and if the cardinality of $\cala$ is $N$ we call it an {\bf $N$-bouquet}  and denote it by $\gB_N(\cald)$. If the $N$-bouquet is connected and all Sasaki cones contain the Reeb field $\xi$, we say that $\gB_N(\cald)$ is an {\bf $N$-bouquet based at $\xi$}. 
\end{definition}

Clearly a $1$-bouquet of Sasaki cones is just a Sasaki cone. As we shall see the Sasaki cones in $\gB(\cald)$ can have varying dimension. Also of interest are the connectivity properties of $\gB(\cald)$ and $\gB_\cala(\cald)$. For example, the $N$-bouquets described in \cite{Boy10b} are connected; however, nothing is known about the connectivity of the complete bouquet. From the construction of the bouquet we have

\begin{lemma}\label{numcones}
The number of cones in a complete bouquet is $\gn_R(\cald)$ while the number of cones of dimension $\gr$ is $\gn_R(\cald,\gr)$.
\end{lemma}

Note that any K-contact structure has a torus of Reeb type in its automorphism group. Since any regular (even quasi-regular) contact structure has a compatible K-contact structure \cite{BG05}, the following theorem generalizes a result of Lerman \cite{Ler02b}. But first I restate a lemma for K-contact structures which was stated and proved for the Sasakian case in \cite{BGS06}. It is easy to see that the proof given there works as well for the more general K-contact case.

\begin{lemma}\label{autS}
Let $(\cald,J)$ be an almost {\rm CR} structure of K-contact type on $M$. For each 
$\cals\in \calk(\cald, J)$, the isotropy subgroup 
$\gA\gu\gt(\cals)\subset \gC\gR(\cald,J)$ at $\cals$ satisfies
$$\bigcap_{\cals \in \calk(\cald, J)} \gA\gu\gt(\cals)= T_k\, ,$$
where $T_k$ is a maximal torus of $\gC\gR(\cald,J)$ and is contained in the isotropy subgroup of every 
$\cals\in  \calk(\cald, J)$.
\end{lemma}

The following is a reformulation and extension of Proposition \ref{reebkcon}.

\begin{theorem}\label{ReebKcon}
If $\gC\go\gn(M,\cald)$ contains a torus $T$ of Reeb type, then  there is an almost complex structure $J$ on $\cald$ and a quasi-regular K-contact structure $\cals=(\xi,\eta,\Phi,g)$ with $\eta$ in $\cald$, $\Phi|_\cald=J$, and $T\subset \gA\gu\gt(\cals).$
Moreover, $T$ is contained in a maximal torus $T_m$ of Reeb type lying in $\gC\go\gn(M,\cald)$, and there is a quasi-regular K-contact structure $\cals'=(\xi',\eta',\Phi',g')$ such that $\ker\eta'=\cald$ and  $T_m\subset \gA\gu\gt(\cals')$. 
\end{theorem}

\begin{proof}
If $\cals$ is quasi-regular, the first statement follows from Proposition \ref{reebkcon}. So assume that $\eta$ is not quasi-regular. Since $(M,\eta)$ is of Reeb type, the Reeb field lies in the Lie algebra $\gt$ of $T,$ and the closure of its leaves is a subtorus $T'$ of $T$ of dimension greater than one. Then the Sasaki cone has dimension larger than one, and the K-contact structure belongs to an underlying almost CR structure $(\cald,J).$ So by the Approximation Theorem \cite{Ruk95a} (see also Theorem 7.1.10 of \cite{BG05}) there is a sequence of quasi-regular K-contact structures $\cals_j$ with underlying almost CR structure $(\cald,J)$ that converges to $\cals.$  Moreover, $T$ is contained in a maximal torus $T_m$ of $\gA\gu\gt(\cals)$ which by Lemma \ref{autS} is the maximal torus for all K-contact structures in the underlying almost CR structure $(\cald,J).$ Since the Reeb vector fields $\xi_j$ of $\cals_j$ all lie in the Sasaki cone associated to $(\cald, J)$, we can take any of the $\cals_j$ as our quasi-regular K-contact structure. Since these are all invariant under the action of $T$, it is a subgroup of $\gA\gu\gt(\cals).$ 

Now $T$ is contained in some maximal torus $T_m$ in $\gC\go\gn(M,\eta)$, and since $T_m$ contains the Reeb vector field it is of Reeb type. Thus, applying the argument above to $T_m$ gives a quasi-regular K-contact structure $\cals'=(\xi',\eta',\Phi',g')$ with underlying contact structure $\cald$ and satisfying $T_m\subset \gA\gu\gt(\cals')$.
\end{proof}

\begin{remark}\label{nonurem}
Example \ref{counter} shows that the conjugacy class of the maximal torus is not unique.
\end{remark}

Theorem \ref{ReebKcon} has an immediate
\begin{corollary}\label{toriReeb}
If $\cald$ is a contact structure of K-contact type, then 
the image of the map $\gQ:\calj(\cald)\ra{1.6} \gS\gC_{T}^{max}(\cald)$ is non-empty and contains $\gS\gC_{RT}^{max}(\cald)$.
\end{corollary}

\begin{remark}
As far as contact invariants are concerned it is worth mentioning the recent result in \cite{NiPa07} which says that any contact structure $\cald$ with $\gn_R(\cald)\geq 1$ must be symplectically fillable.
\end{remark}

\section{The Symplectomorphism Group of a Symplectic Orbifold}\label{symgroup}

\subsection{A Brief Review of Orbifolds}
Let us recall the definition of a diffeomorphism of orbifolds. Here I follow \cite{BG05} which in turn relies heavily on \cite{MoPr97}, and refer to these references for the precise definition of an orbifold. Suffice it to say that an {\it orbifold} $\calx$ is a paracompact Hausdorff space $X$ together with an atlas of local uniformizing charts $(\tU_i,\grG_i,\varphi_i)$ where $\grG_i$ is a finite group acting effectively on the local uniformizing neighborhood $\tU_i\approx \bbr^n$ and $\varphi_i:\tU_i\ra{1.4} \varphi(\tU_i)$ induces a homeomorphism between $\tU_i/\grG_i$ and $U_i=\varphi_i(\tU_i).$ Moreover, these charts patch together in an appropriate way, that is, given two such charts $(\tU_i,\grG_i,\varphi_i)$ and  $(\tU_j,\grG_j,\varphi_j)$ there are smooth embeddings $\grl_{ij}:\tU_i\ra{1.4} \tU_j$, called {\it injections} that satisfy $\varphi_j\circ \grl_{ij} =\varphi_j.$

\begin{definition}\label{orbdiff}
Let $\calx=(X,\calu)$ and $\caly=(Y,\calv)$ be orbifolds. A map
$f:X\ra{1.3} Y$ is said to be {\bf smooth} if for every point
$x\in X$ there are orbifold charts $(\tU_i,\grG_i,\varphi_i)$ about $x\in U_i=\varphi(\tU_i)$
and $(\tV_j,\grF_j,\psi_j)$ about $f(x)\in V_j=\psi(\tV_j)$ such that $f_{ij}=f|_{U_i}$ maps
$U_i$ into $V_j$ and there exist local lifts
$\tf_{ij}:\tU_i\ra{1.3} \tV_j$ satisfying $\psi_j\circ \tf_{ij}
=f_{ij}\circ \varphi_i.$ An {\bf orbifold diffeomorphism} between $\calx=(X,\calu)$ and $\caly=(Y,\calv)$ consists of smooth maps $f:X\ra{1.3} Y$ and $g:Y\ra{1.3} X$ such that 
$f\circ g=\BOne_Y$ and $g\circ f=\BOne_X,$ where $\BOne_X,\BOne_Y$ are the
respective identity maps. 
\end{definition}

One problem with this definition is that there is an ambiguity concerning the choice of lifts, and correspondingly the choice of homomorphism $(\tf_{\tU})_*:\grG\ra{1.3} \Phi.$ Considering different choices of lifts as giving different orbifold maps gives rise to what is referred to in \cite{BoBr02,BoBr06} as {\it unreduced orbifold maps}. Whereas, if one considers just the maps $f$ as long as smooth lifts exists, one obtains the set of {\it reduced orbifold maps}. The latter is the usual interpretation of orbifold maps, and this is the convention that I follow. 

\begin{example}\label{diffex1} 
{\rm An orbifold diffeomorphism is generally a stronger notion than a diffeomorphism of manifolds. For example, consider the weighted projective lines $\bbc\bbp(p,q)=S^2(p,q)$ where $p$ and $q$ are relatively prime integers satisfying $p<q.$ Each has a natural complex orbifold structure. But for any such $p$ and $q$, $\bbc\bbp(p,q)$ is isomorphic as smooth projective algebraic varieties to the projective line $\bbc\bbp^1$, so they are all diffeomorphic as smooth manifolds to the 2-sphere $S^2$. However, they are {\it not} diffeomorphic as real orbifolds.}
\end{example} 

Let us consider the case $\caly=\calx$. Then we define
\begin{equation}\label{orbdiffgp}
\gD\gi\gf\gf(\calx)=\{f:\calx\ra{1.5} \calx ~|~f~{\rm  is~an~orbifold~diffeomorphism}\}
\end{equation}
and the `unreduced' group $\gD\gi\gf\gf_u(\calx)$ is defined in the same way, except now one considers different local lifts as giving different maps $f.$ It is easy to see that both $\gD\gi\gf\gf(\calx)$ and $\gD\gi\gf\gf_u(\calx)$ form a group under composition of maps.  So when $\calx$ is compact there is a finite normal subgroup $\gG$ of $\gD\gi\gf\gf_u(\calx)$ consisting of the set of all lifts of the identity map. An original choice of lifts in each uniformizing neighborhood defines the identity in $\gG$ and $\gD\gi\gf\gf_u(\calx)$ so that $\gG|_{\tU_i}=\grG_i$. So we have an exact sequence of groups \cite{BoBr06}
\begin{equation}\label{redunred}
\BOne\ra{1.8}\gG\ra{1.8} \gD\gi\gf\gf_u(\calx)\ra{1.8} \gD\gi\gf\gf(\calx)\ra{1.8} \BOne.
\end{equation}
Moreover, again when $\calx$ is compact, $\gD\gi\gf\gf_u(\calx)$ has the structure of a Fr\'echet manifold \cite{BoBr02,BoBr06} modelled on the space of smooth sections of the tangent orbibundle. Recall that orbifolds have a natural stratification determined by their local isotropy groups, and $\gD\gi\gf\gf(\calx)$ preserves the strata.

\begin{example}\label{diffex2}
Continuing with Example \ref{diffex1} take $\calx=S^2(p,q)$ and consider the unreduced diffeomorphism group $\gD\gi\gf\gf_u(S^2(p,q))$. For each of the $pq$ choices of lifts, we have a diffeomorphism of $S^2$ which fixes the two singular points, say the north and south pole. So $\gD\gi\gf\gf(S^2(p,q))$ is the quotient group $\gD\gi\gf\gf_u(S^2(p,q))/\bbz_{pq}$ which can be identified with a subgroup of $\gD\gi\gf\gf(S^2)$ that fixes these two points. This is the reduced orbifold group in the terminology of Borzellino and Brunsden, cf. Corollary 2 of \cite{BoBr06}.
\end{example}

\subsection{Symplectic Orbifolds}
Symplectic orbifolds seem to have first been considered by Weinstein \cite{Wei77b}.

\begin{definition}
A {\bf symplectic orbifold} $(\calx,\gro)$ is an orbifold $\calx$ together with a symplectic form $\gro_i$ on each local uniformizing neighborhood $\tU_i$ that is invariant under $\grG_i$ and satisfies $\grl_{ji}^*(\gro_j|_{\grl_{ji}(U_i)})=\gro_i.$ 
\end{definition}

Let $(\calx,\gro)$ be a symplectic orbifold, then we define the {\it group of orbifold symplectomorphisms} by
\begin{equation}\label{sympgrp}
\gS\gy\gm(\calx,\gro)=\{f\in \gD\gi\gf\gf(\calx) ~|~f^*\gro=\gro\}.
\end{equation}
This just means that if $\tf_{ij}:\tU_i\ra{1.5} \tU_j$ denotes the local lift of $f$ to $\tU_i,$ we have $\tf_{ij}^*\gro_j=\gro_i.$ 
Since working on orbifolds means working equivariantly with sheaves or bundles on the local uniformizing covers, and invariantly with their sections, most of the standard results for symplectic manifolds also hold for symplectic orbifolds, so we shall only give details when differences occur. Of course, we also have the unreduced group $\gS\gy\gm_u(\calx,\gro)$ where lifts of the identity give rise to different maps. In this case one works with smooth sections $\gro_i$, not necessarily invariant ones. We need to assure ourselves that Weinstein's local charts (cf. \cite{McDSa}, pg 103) for the symplectomorphism group works equally well for orbifolds. Recall that his proof involves identifying a neighborhood of the zero section in $T^*M$ with a neighborhood of the diagonal in $M\times M$ regarding $M$ as a Lagrangian submanifold of $M\times M$ with symplectic form $(-\gro)\oplus \gro.$ This follows in turn by identifying the normal bundle of $M$ in $M\times M$ with the tangent bundle $TM.$ But the construction of normal bundles and the identification of a neighborhood of a suborbifold with a neighborhood of the zero section of its normal bundle are natural, and as discussed in \cite{LeTo97} works equally well for orbifolds.

Summarizing we have
\begin{proposition}\label{Weincharts}
Let $(\calb,\gro)$ be a compact symplectic orbifold. Then there is a diffeomorphism between a neighborhood of the identity in $\gS\gy\gm(\calx,\gro)$ and a neighborhood of $0$ in the vector space of closed $\gG$-invariant sections of the orbibundle $T^*\calb$. 
\end{proposition}

\begin{proof}
Given the discussion above, the proof in \cite{McDSa} goes through for orbifolds with one caveat. We need to be mindful of the fact that it is the unreduced group that is modelled on the vector space $Z^1(\calb)$ of closed 1-forms of the orbibundle $T^*\calb.$ The group $\gG$ in the exact sequence (\ref{redunred}) is the direct product of the local uniformizing groups $\grG$ on each uniformizing neighborhood. Since $\calb$ is compact $\gG$ is finite, and on the vector space level we conclude that the reduced group $\gS\gy\gm(\calx,\gro)$ is modelled on the vector space $Z^1(\calb)^\gG$ of $\gG$-invariant sections.
\end{proof}

Then as in the manifold case we have

\begin{corollary}\label{orbWeinstein}
Let $(\calb,\gro)$ be a compact symplectic orbifold. Then the group $\gS\gy\gm(\calb,\gro)$ is locally path-connected.
\end{corollary}

Let $\gS\gy\gm_0(\calb,\gro)$ denote the subgroup of $\gS\gy\gm(\calb,\gro)$ connected to the identity.
It follows from Corollary \ref{orbWeinstein} that for each $f\in \gS\gy\gm_0(\calb,\gro)$ there is a smooth family $\{f_t\}\subset \gS\gy\gm_0(\calb,\gro)$ satisfying $f_0=\BOne$ and $f_1=f.$ Moreover, there are smooth vector fields $X_t$ (smooth invariant sections of the orbibundle $T\calb$) such  that 
\begin{equation}\label{symisotopy}
 \frac{d}{dt}f_t=X_t\circ f_t.
\end{equation}
Since the $f_t$ are symplectomorphisms the vector fields $X_t$ are symplectic, that is, satisfy $d(X_t\hook \gro)=0.$ Equation (\ref{symisotopy}) also holds in the case of smooth (not necessarily invariant) sections of $T\calb$ in which case $f_t$ are elements of the unreduced group.

\subsection{Hamiltonian Isotopies}
If the 1-forms $X_t\hook \gro$ are exact for all $t$ we have smooth functions $H_t$, called {\it Hamiltonians}, that satisfy $X_t\hook \gro=dH_t.$ The family $\{f_t\}$ is then called a {\it Hamiltonian isotopy}. A symplectomorphism $f\in \gS\gy\gm(\calb,\gro)$ is called {\it Hamiltonian} if there is a Hamiltonian isotopy $f_t$ such that $f_0=\BOne$ and $f_1=f.$ We denote by $\gH\ga\gm(\calb,\gro)$ the subset of all Hamiltonian symplectomorphisms. Moreover, the set of vector fields $X_H$ satisfying $X_H\hook \gro=dH$ forms a Lie subalgebra $\gh\ga\gm(\calb,\gro)$ of $\gs\gy\gm(\calb,\gro),$ called the {\it Lie algebra of Hamiltonian vector fields}. It is easy to see that  $\gH\ga\gm(\calb,\gro)$ is a normal subgroup of $\gS\gy\gm_0(\calb,\gro)$. In the manifold case there is a well-known exact sequence 
\begin{equation}\label{fluxdef}
0\ra{1.8} \gH\ga\gm(B,\gro)\ra{1.8} \gS\gy\gm_0(B,\gro)\fract{\gF\gl}{\ra{1.8}} H^1(B,\bbr)/\grG_\gro \ra{1.8} 0
\end{equation}
for some discrete countable subgroup $\grG,$ and it is easy to see that this also holds for compact symplectic orbifolds. The construction of $\gF\gl$ passes to the universal covering groups and uses the fact that such groups are realized by homotopy classes of smooth paths with fixed endpoints (see \cite{Ban97,McDSa} for details). The discreteness of $\grG_\gro$ is one formulation of the so-called Flux Conjecture  which has been recently proved by Ono \cite{Ono06}. It seems plausible that this also holds in the orbifold case, but we do not need it, since it is known to hold when the cohomology class $[\gro]$ is rational. So we need only verify it for compact symplectic orbifolds with a rational cohomology class. The importance of the Flux Conjecture is that it is equivalent to the group $\gH\ga\gm(\calb,\gro)$ being $C^1$-closed in $\gS\gy\gm_0(\calb,\gro)$.

\begin{proposition}\label{flux}
Let $(\calb,\gro)$ be a compact symplectic orbifold with $[\gro]\in H^2(B,\bbq)$. Then the group $\grG_\gro$ is discrete, and $\gH\ga\gm(\calb,\gro)$ is $C^1$-closed in $\gS\gy\gm_0(\calb,\gro)$ and a Fr\'echet Lie subgroup whose Lie algebra is $\gh\ga\gm(\calb,\gro)$.
\end{proposition}

\begin{proof}
As in  the manifold case the proof of the first statement follows from the definition of the flux map $\gF\gl$ (see pg 324 of \cite{McDSa}). The proof of the second statement is essentially the same as in Proposition 10.20 of \cite{McDSa} with the only caveat being that we work with invariant forms on the local uniformizing neighborhoods of the orbifold.
\end{proof}

\begin{proposition}\label{WeinHam}
Let $(\calb,\gro)$ be a compact symplectic orbifold with $[\gro]\in H^2(\calb,\bbq)$. Then there is a diffeomorphism between a neighborhood of the identity in $\gH\ga\gm(\calb,\gro)$ and a neighborhood of $0$ in the vector space $B^1(\calb)^\gG$ of $\gG$-invariant exact sections of the orbibundle $T^*\calb$. 
\end{proposition}

\begin{proof}
Using $\gro$ we can identify  $B^1(\calb)^\gG$ with $\gh\ga\gm(\calb,\gro)$. Since the class $[\gro]$ is rational, Proposition \ref{flux} implies that $\gH\ga\gm(\calb,\gro)$ is a closed Lie subgroup of $\gS\gy\gm_0(\calb,\gro)$, and thus, an embedded submanifold. So the Weinstein charts for $\gH\ga\gm(\calb,\gro)$ are obtained by intersecting the Weinstein charts of $\gS\gy\gm_0(\calb,\gro)$ with $\gH\ga\gm(\calb,\gro)$.
\end{proof}

\section{Contactomorphisms and Symplectomorphisms}\label{consymgr}

\subsection{K-Contact Structures and Symplectic Orbifolds}
Now suppose that $\cals=(\xi,\eta,\Phi,g)$ is a quasi-regular K-contact structure on a compact manifold $M$, then by Theorem \ref{bgthm} the space of leaves $M/\calf_\xi=B$ has the structure of an almost K\"ahler orbifold $(\calb,\gro)$ such that $M$ is the total space of an $S^1$ orbibundle over $\calb$ with K\"ahler form $\gro$ satisfying $\pi^*\gro=d\eta.$ Let $\phi\in \gS\gy\gm(\calb,\gro)$ be a symplectomorphism, so that $\phi^*\gro=\gro.$ However, as in Corollary 8.1.9 of \cite{BG05} it is only the Hamiltonian symplectomorphisms $\gH\ga\gm(\calb,\gro)$ that lift to contactomorphisms of $M.$  The next theorem generalizes a result of Banyaga \cite{Ban78}.

\begin{theorem}\label{consymiso}
Let $\cals=(\xi,\eta,\Phi,g)$ be a quasi-regular K-contact structure on a compact manifold $M$. Then there is an isomorphism of Fr\'echet Lie groups 
$$\gC\go\gn(M,\eta)_0/\Xi\approx \gH\ga\gm(\calb,\gro),$$ 
where $\Xi\approx S^1$ is the central one-parameter subgroup generated by the Reeb vector field.
\end{theorem}

\begin{proof}
It is well known that $\Xi$ is central, and since $\cals$ is quasi-regular, it is a closed Lie subgroup of $\gC\go\gn(M,\eta)_0$ which by Proposition \ref{etasub} is a Fr\'echet Lie group. So we have an exact sequence of Lie groups
$$\BOne\ra{1.8} \Xi\ra{1.8} \gC\go\gn(M,\eta)_0\ra{1.8} \gC\go\gn(M,\eta)_0/\Xi\ra{1.8} \BOne.$$
I want to identify $\gC\go\gn(M,\eta)_0/\Xi$ with $\gH\ga\gm(\calb,\gro).$
Let $f\in \gC\go\gn(M,\eta)_0/\Xi$ and consider the orbifold submersion $\pi:M\ra{1.4} B$ whose fibres are the $S^1$ orbits of $\Xi.$  Define a map $\bar{f}:B\ra{1.5} B$ by $\pi\circ f= \bar{f}\circ \pi.$ Since $\Xi$ is central, $\bar{f}$ is well defined, and it is easy to check that the map $\nu(f)= \bar{f}$ is injective and a group homomorphism. I claim that $\bar{f}\in \gH\ga\gm(\calb,\gro)$. First we show that $\bar{f}$ is an orbifold diffeomorphism. Let $x_0\in M$ be a point with isotropy subgroup $G_{x_0}$ under the $S^1$ action. By the slice theorem there is a neighborhood of $x_0$ of the form $\tU\times S^1$ where $\tU$ is $G_{x_0}$-invariant. (Note that since $G_{x_0}$ is a subgroup of $S^1$ it is necessarily cyclic). Then $(\tU,G_{x_0},\varphi)$ is a local uniformizing neighborhood for the orbifold $\calb$, where $\varphi$ is defined by $\varphi(\tx)=\pi(\tx,t)=\pi(x).$ Since $f$ is $S^1$-invariant its restriction to $\tU\times S^1$ provides the necessary smooth lift of $\bar{f}.$ This shows that $\bar{f}$ is a smooth orbifold map. But since $f$ is a diffeomorphism we can apply the same reasoning to $f^{-1}$ showing that $\bar{f}$ is an orbifold diffeomorphism.
To see that $\nu(f)=\bar{f}\in \gS\gy\gm(\calb,\gro)_0$ we check that $\bar{f}^*\gro_i=\gro_i$ on each uniformizing neighborhood. This follows easily from the relations $f^*\eta=\eta$, the definition of $\bar{f}$, and $\pi^*\gro=d\eta.$ 

To proceed further we consider the Lie algebras.
\begin{lemma}\label{Liealgiso}
There is an isomorphism of Lie algebras $\gc\go\gn(M,\eta)/\gg_\xi\approx \gh\ga\gm(\calb,\gro).$
\end{lemma}

\begin{proof}
A choice of contact 1-form $\eta$ gives a well-known isomorphism $X\mapsto \eta(X)$ between $\gc\go\gn(M,\cald)$ and the smooth functions $C^\infty(M)$ with Jacobi bracket defined by $\{f,g\}=\eta([X_f,X_g]).$ Under this isomorphism the subalgebra $\gc\go\gn(M,\eta)$ maps to the $\Xi$-invariant smooth functions $C^\infty(M)^\Xi$, and the subalgebra $\gg_\xi\subset \gc\go\gn(M,\eta)$ maps to the constant functions. But there is an isomorphism between $C^\infty(M)^\Xi$ and the smooth functions $C^\infty(\calb)$ on the orbifold $\calb.$ To see this we remark that local uniformizing groups of $\calb$ are precisely the isotropy groups of the $\Xi$ action on $M,$ the $\Xi$ invariant smooth functions on $M$ push down to functions on $B$ with smooth lifts to the local uniformizing neighborhoods $\tU.$ Conversely, it is precisely these functions that pullback to $\Xi$-invariant functions on $M.$ So we also have isomorphisms 
$$\gc\go\gn(M,\eta)/\gg_\xi\approx C^\infty(M)^\Xi/\{\rm constants\}\approx C^\infty(\calb)/\{\rm constants\}\approx \gh\ga\gm(\calb,\gro),$$
which proves the result.
\end{proof}

Continuing with the proof of Theorem \ref{consymiso}, we see from the lemma that the Fr\'echet Lie group $\gC\go\gn(M,\eta)_0/\Xi$ is modelled on  $\gc\go\gn(M,\eta)/\gg_\xi\approx \gh\ga\gm(\calb,\gro).$ So there is a local isomorphism of Fr\'echet Lie groups $\gC\go\gn(M,\eta)_0/\Xi\approx \gH\ga\gm(\calb,\gro).$ But since the map $\nu:\gC\go\gn(M,\eta)_0/\Xi \ra{1.5} \gS\gy\gm(\calb,\gro)_0$ is a monomorphism of Fr\'echet Lie groups and $\gH\ga\gm(\calb,\gro)$ is connected the result follows.
\end{proof}

Alternatively, we say that $\gC\go\gn(M,\eta)_0$ is an $S^1$ central extension of $\gH\ga\gm(\calb,\gro)$, that is, we have an exact sequence of groups:
\begin{equation}\label{Hamexact}
\BOne\ra{1.8} \Xi\ra{1.8} \gC\go\gn(M,\eta)_0\fract{\grr}{\ra{1.8}} \gH\ga\gm(\calb,\gro)\ra{1.8} \BOne.
\end{equation}

\subsection{Lifting Tori} Given a quasi-regular K-contact structure on $M$, we want to be able to lift Hamiltonian tori to obtain tori in the contactomorphism group of $M$.
 
\begin{proposition}\label{maxtorconsym}
Let $\cals=(\xi,\eta,\Phi,g)$ be a quasi-regular K-contact structure on a compact manifold $M$. Then
any torus $T\subset \gH\ga\gm(\calb,\gro)$ lifts to a torus $\grr^{-1}(T)$ in $\gC\go\gn(M,\eta)_0$ where $\grr$ is the map defined by the exact sequence {\rm(\ref{Hamexact})}. Furthermore, $T$
is a maximal torus in $\gH\ga\gm(\calb,\gro)$ if and only if $\grr^{-1}(T)\times \Xi$ is a maximal torus in $\gC\go\gn(M,\cald)$.
\end{proposition}

\begin{proof}
The proof of the first statement is essentially that given in \cite{Ler02b}. 
Let $T$ be a Lie subgroup of $\gH\ga\gm(\calb,\gro)$ that is isomorphic as an abstract Lie group to a finite dimensional torus. We know how to lift vector fields in the Lie algebra $\gt$ of $T$ at least in the case of smooth manifolds. But the same method works for orbifolds since vector fields in $ \gh\ga\gm(\calb,\gro)$ are invariant under the local uniformizing groups, that is under $\gG$. (The action of tori on orbifolds has been studied by Haefliger and Salem \cite{HaSa91}). So let $X\in \gt$ and look for a vector field $\tX\in \gc\go\gn(M,\eta)$ that satisfies
$$0=\pounds_{\tX}\eta=\tX\hook d\eta + d(\eta(\tX))=X^h\hook d\eta + d(\eta(\tX))$$ 
where $X^h$ is the horizontal lift of $X.$ Generally there is an obstruction to solving this which lies in the basic cohomology group $H^1(\calf_\xi).$ However, since $X$ is Hamiltonian this obstruction vanishes. Indeed, we have $X\hook\gro=-dH$ for some $\gG$-invariant smooth function $H.$ But also $X^h\hook d\eta=\pi^*(X\hook\gro)=-\pi^*dH=-d\pi^*H$. So we get a solution by taking $\eta(\tX)=\pi^*H.$ Let $\{X_i\}_i$ be a basis of for $\gt$ consisting of periodic Hamiltonian vector fields and let $H_i$ be their Hamiltonians. Since the $X_i$s commute the $H_i$s commute under the Poisson bracket. But then the corresponding functions $\eta(\tX_i)$s commute under the Jacobi-Poisson bracket on $M.$ Thus, their lifts $\tX_i$ commute as well. Furthermore, since $\eta(\tX_i)$ is basic for each $i$, the $X_i$s commute with the Reeb field $\xi.$ So the lifted Lie algebra spanned by $\{\tX_i,\xi\}_i$ is Abelian. It remains to show that the $\tX_i$s are periodic vector fields on $M.$ So let $X\in \gt$ be periodic with period $\grt$ and fix a point $p\in M$. The loop traced out by $X$ in $B$ is $\mathscr{L}(t)=\{{\rm exp} tX ~|~t\in [0,\grt]\}$ with $\mathscr{L}(0)=\mathscr{L}(\grt)=\pi(p).$ Lifting this loop to $M$ gives a path $\tilde{\mathscr{L}}(t)=\{{\rm exp} tX^h ~|~t\in [0,\grt]\}$ begining at $p$ and ending at a point $p'\in \pi^{-1}(\pi(p)).$ Since the endpoint is in the same fibre as $p$ we can flow using the Reeb vector field from $p'$ to $p.$ But this can be accomplished by $\eta(\tX)(\pi(p'))\xi$ since $\eta(\tX)$ is basic so it is constant along the flow of $\xi$ and $\eta(\tX)\xi$ commutes with $X^h$. Since $p$ is arbitrary and $\eta(\tX)$ is a smooth basic function, $\tX$ is periodic, and generates a smooth circle action on $M$ that preserves $\eta.$ Hence, $\grr^{-1}(T)$ is a torus in  $\gC\go\gn(M,\eta)_0$ commuting with the Reeb flow. Moreover, if $T$ is maximal in $\gH\ga\gm(\calb,\gro)$ then $\grr^{-1}(T)\times \Xi$ is maximal in $\gC\go\gn(M,\eta)_0$, and thus, maximal in $\gC\go\gn(M,\cald)$ by Theorem \ref{ReebKcon}.
\end{proof}

Recall that $S^1$-orbibundles on an orbifold $\calb$ are classified by elements of Haefliger's orbifold cohomology group $H^2_{orb}(\calb,\bbz)=H^2(B\calb,\bbz)$ where $B\calb$ is the classifying space of the orbifold and gives rise to a natural projection $p:B\calb\ra{1.6} \calb$ (cf. Section 4.3 of \cite{BG05} for details). We now have

\begin{theorem}\label{contor}
Let $(\calb,\gro)$ be a compact symplectic orbifold with $[p^*\gro]\in H^2_{orb}(\calb,\bbz),$ and let $\pi:M\ra{1.5} \calb$ be the $S^1$ orbibundle defined by $[\gro]$. Furthermore, let $\eta$ be a connection 1-form on $M$ such that $\pi^*\gro=d\eta$ and define $\cald=\ker\eta.$ Then
two maximal tori $T_1$ and $T_2$ are conjugate in $\gH\ga\gm(\calb,\gro)$ if and only if the lifted maximal tori $\grr^{-1}(T_1)\times \Xi$ and $\grr^{-1}(T_2)\times \Xi$ are conjugate in $\gC\go\gn(M,\cald)$.
\end{theorem}

\begin{proof}
The only if part is an immediate consequence of Proposition \ref{maxtorconsym} and Theorem \ref{consymiso}. Now assume that $T_1$ and $T_2$ are not conjugate in $\gH\ga\gm(\calb,\gro)$, and suppose there is a $\phi\in \gC\go\gn(M,\cald)$ such that 
$${\rm Ad}_{\phi^{-1}}(\grr^{-1}(T_1)\times \Xi)=\grr^{-1}(T_2)\times \Xi,$$ 
where $\eta$ is the contact form in $\gC(\cald)$ whose Reeb vector field is $\xi.$ In terms of the Lie algebras we have 
\begin{equation}\label{conliealg}
 {\rm Ad}_{\phi^{-1}}(\grr^{-1}(\gt_1)\times \gg_\xi)=\grr^{-1}(\gt_2)\times \gg_\xi.
\end{equation}
Choose bases $X_0=\xi,X_1,\cdots,X_k$ for $\grr^{-1}(\gt_1)\times \gg_\xi$ and $Y_0=\xi,Y_1,\cdots,Y_k$ for $\grr^{-1}(\gt_2)\times \gg_\xi,$ respectively. Then Equation (\ref{conliealg}) says there are real numbers $a_{ij}$ such that 
$${\rm Ad}_{\phi^{-1}}(X_i)=\sum_j a_{ij}Y_j.$$
In terms of the moment map this implies
$$\eta({\rm Ad}_{\phi^{-1}}(X_i))=\eta(\sum_j a_{ij}Y_j)=\sum_j a_{ij}\eta(Y_j)=a_{i0} +\sum_{j=1}^k a_{ij}\eta(Y_j).$$
On the other hand the left hand side of this equation equals $(\phi^*\eta)(X_i)=f_\phi\eta(X_i).$
Setting $i=0$ gives $f_\phi=a_{00} +\sum_{j=1}^k a_{0j}\eta(Y_j),$ and the coefficients must be given such that this is everywhere positive. So the moment maps for the two tori are related by
\begin{equation}\label{mommaptor}
\eta(X_i)=\frac{a_{i0} +\sum_{j=1}^k a_{ij}\eta(Y_j)}{a_{00} +\sum_{j=1}^k a_{0j}\eta(Y_j)},\qquad i=1,\cdots,k.
\end{equation}
Since the $X_i$ and $Y_i$ all commute with $\xi$, all these moment maps are basic functions on $M.$
But also $Y_jf_\phi=0$ for all $j$ since the $Y_i$s commute, and $Y_j\in \gc\go\gn(M,\eta)$. But then Equation (\ref{mommaptor}) implies that $Y_j\eta(X_i)=0$ for all $i,j,$ and this implies that $\eta([Y_j,X_i])=0$ for all $i,j.$ But then Lemma \ref{Dcon} implies that $[Y_j,X_i]=0$, which by maximality implies that $\gt_1=\gt_2.$ This contradicts the fact that $T_1$ and $T_2$ are not conjugate. 
\end{proof}

Recall that to each compatible almost complex structure $J\in \calj(\cald)$ on a contact manifold of K-contact type we associate a conjugacy class $\calc_T(J)\in \sqcup_{\gr\geq 1}\gS\gC_T(\cald,\gr)$ of maximal tori in $\gC\go\gn(M,\cald).$
Furthermore, a choice of almost complex structure $J$ on $M$ gives an almost complex structure $\hat{J}$ on $\calb.$ We are interested in the converse. Let $(\calb,\gro)$ be a symplectic orbifold with $[\gro]\in H^2_{orb}(\calb,\bbz)$ and cyclic local uniformizing groups. Let $M$ be the total space of the corrresponding $S^1$ orbibundle over $\calb$ obtained by the orbifold Boothby-Wang construction, and assume that the local uniformizing groups inject into the $S^1.$ As in the contact case we denote by $\gS\gC_T(\calb,\gro)$ the set of conjugacy classes of maximal tori in the group of Hamiltonian diffeomorphisms $\gH\ga\gm(\calb,\gro)$, and by $\gS\gC_T(\calb,\gro;\gr)$ the subset of conjugacy classes of maximal tori of dimension $\gr$. We also denote the cardinality of $\gS\gC_T(\calb,\gro)$ by $\gn(\gro)$, and the cardinality of $\gS\gC_T(\calb,\gro; \gr)$ by $\gn(\gro,\gr)$, respectively. Notice that if $\cald$ and $\gro$ are related by the orbifold Boothby-Wang construction then $\gn_R(\cald)\geq \gn(\gro)$. Then from Proposition \ref{maxtorconsym} and Theorem \ref{contor} we have

\begin{corollary}\label{symconT}
There is an injection $\gri:\gS\gC_T(\calb,\gro;\gr)\ra{1.8} \gS\gC_T(\cald,\gr+1)$.
\end{corollary}

Now suppose that $(\calb,\gro)$ has a $T$-invariant almost complex structure $\hat{J}$ that is compatible with $\gro$ in the sense that $\gro(\hat{J}X,Y)$ defines a Riemannian metric on $\calb.$ This gives the orbifold $\calb$ an almost K\"ahler structure which by Theorem \ref{kconinversionthm} lifts to a K-contact structure $(\xi,\eta,\Phi,g)$ on $M$ where $d\eta=\pi^*\gro,$ and $\Phi|_\cald=J$ which is the horizontal lift of $\hat{J}.$ In analogy with the Section \ref{secalmostcomp} we define $\calj(\gro)$ to be the set of almost complex structures $\hat{J}$ on the symplectic orbifold $(\calb,\gro)$ that are compatible with $\gro.$ Let $\gA\gu\gt(\gro,\hat{J})$ denote the automorphism group of the almost K\"ahler structure. Since $\calb$ is compact, so is $\gA\gu\gt(\gro,\hat{J})$. Let  $\calj(\gro,\gr)$ denote the subset of almost complex structures such that $\gA\gu\gt(\gro,\hat{J})$ has rank $\gr.$ Then we have a function $\gQ_{\gro,\gr}:\calj(\gro,\gr)\ra{1.5} \gS\gC_T(\calb,\gro;\gr)$ together with a commutative diagram
\begin{equation}\label{commdiag}
\begin{matrix}
\calj(\cald,\gr+1) &\fract{\gQ_{\cald,\gr}}{\ra{2.8}}& \gS\gC_T(\cald,\gr+1)\\
         \uparrow && \uparrow \\
         \calj(\gro,\gr) &\fract{\gQ_{\gro,\gr}}{\ra{2.8}}& \gS\gC_T(\calb,\gro;\gr) \\
        \uparrow && \uparrow \\
         0 && 0 &.
\end{matrix}
\end{equation}
Summarizing we have
\begin{theorem}\label{cardReeb}
Let $(\calb,\gro)$ be a compact symplectic orbifold with $N$ conjugacy classes of maximal tori in its group of Hamiltonian isotopies $\gH\ga\gm(\calb,\gro)$. Suppose further that $\gro$ defines an integral class in $H^2_{orb}(\calb,\bbz)$ and that the corresponding $S^1$-orbibundle has a smooth total space $M$. Let $\eta$ be a contact 1-form obtained by the orbifold Boothby-Wang construction and let $\xi$ denote its Reeb vector field. Then the contact manifold $(M,\cald)$ where $\cald=\ker\eta$ has an $N$-bouquet $\gB_N(\cald)$ of Sasaki cones based at $\xi$ where $N=\gn(\gro)$.
\end{theorem}

As we shall see by example in Section \ref{vardim} the cones can have varying dimensions.

\subsection{The Join Construction and Tori}
First I briefly recall the join construction in \cite{BGO06} which assumed that the structures were Sasakian, but the construction works equally well for K-contact structures. In fact, it is convenient to ``forget'' any K-contact or Sasakian structure, and just think in terms of quasi-regular contact structures, since that is all that is needed for the construction.
We denote by $\calk\calo$ the set of
compact quasi-regular K-contact orbifolds, by $\calk\calm$ the
subset of $\calk\calo$ that are smooth manifolds, and by $\calr\subset
\calk\calm$ the subset of compact, simply connected, regular K-contact manifolds.
The set $\calk\calo$ is topologized with the $C^{m,\gra}$ topology, and the
subsets are given the subspace topology. The set $\calk\calo$ is graded by dimension, that is,
$$\calk\calo =\bigoplus_{n=0}^\infty \calk\calo_{2n+1},$$
and similarly for $\calk\calm$ and $\calr.$  

For each pair of relatively prime positive integers $(k_1,k_2)$ we define a graded
multiplication
\begin{equation}\label{joinmaps}
\star_{k_1,k_2}:\calk\calo_{2n_1+1}\times \calk\calo_{2n_2+1}\ra{1.5}
\calk\calo_{2(n_1+n_2)+1}
\end{equation}
as follows: Let $\cals_1,\cals_2\in \calk\calo$ of dimension $2n_1+1$ and
$2n_2+1$ respectively. Since each orbifold $\cals_i$ has a quasi-regular K-contact
structure,
its Reeb vector field generates a locally free circle action, and the quotient space $\calb_i$ by this
action has a natural orbifold structure. Thus, there is a locally free
action of the 2-torus $T^2$ on the product orbifold $\cals_1\times \cals_2,$ and the
quotient orbifold is the product of orbifold $\calb_1\times \calb_2.$ (Locally free torus actions on orbifolds have been studied in \cite{HaSa91}). Now the contact 1-form on $\cals_i$ determines a
symplectic form $\gro_i$ on the orbifold $\calb_i$, but in order to obtain an integral orbifold
cohomology class $[\gro_i]\in H^2(\calb_i,\bbz)$ we need to assure that the period of a
generic orbit is one. By a result of Wadsley \cite{Wad} the period function on a
quasi-regular K-contact orbifold is lower semi-continuous and constant on the dense open set of regular
orbits. This is because on such an orbifold all Reeb orbits are geodesics. Thus, by a
transverse homothety we can normalize the period function to be the constant $1$ on the
dense open set of regular orbits. In this case the symplectic forms $\gro_i$ define integer
orbifold cohomology classes $[\gro_i]\in H^2_{orb}(\calb_i,\bbz).$ This is the orbifold cohomology defined by Haefliger \cite{Hae84} (see also \cite{BG00a}). Now each
pair of positive integers $k_1,k_2$ gives a symplectic form $k_1\gro_1+k_2\gro_2$ on the
product. Furthermore,  $[k_1\gro_1+k_2\gro_2]\in H^2_{orb}(\calb_1\times \calb_2,\bbz),$
and thus defines an $S^1$ V-bundle over the orbifold $\calb_1\times \calb_2$ whose total 
space is an orbifold that we denote by $\cals_1\star_{k_1,k_2} \cals_2$ and refer to as the
$(k_1,k_2)$-{\it join} of $\cals_1$ and $\cals_2.$  By choosing a connection
1-form $\eta_{k_1,k_2}$ on $\cals_1\star_{k_1,k_2} \cals_2$ whose curvature is
$\pi^*(k_1\gro_1+k_2\gro_2)$, we obtain a quasi-regular contact structure which is unique up to a gauge
transformation of the form $\eta\mapsto \eta +d\psi$ where $\psi$ is a smooth basic
function. This defines the maps in (\ref{joinmaps}). When $\cals_i$ are quasi-regular contact
structures on the compact manifolds $M_i,$ respectively, we shall use the notation $M_1\star_{k_1,k_2} M_2$ instead of $\cals_1\star_{k_1,k_2} \cals_2$. Notice also
that if $\gcd(k_1,k_2)=m$ and we define $(k'_1,k'_2)=(\frac{k_1}{m},\frac{k_2}{m}),$ then
$\gcd(k'_1,k'_2)=1$ and $M_1\star_{k_1,k_2} M_2\approx (M_1\star_{k'_1,k'_2}
M_2)/\bbz_m.$ In this case the cohomology class $k'_1\gro_1+k'_2\gro_2$ is indivisible in
$H^2_{orb}(\calb_1\times \calb_2,\bbz).$ Note also that $M_1\star_{k_1,k_2} M_2$ can be
realized as the quotient space $(M_1\times M_2)/S^1(k_1,k_2)$ where the $S^1$   action
is given by the map
\begin{equation}\label{s1action}
(x,y)\mapsto (e^{ik_2\theta}x,e^{-ik_1\theta}y).
\end{equation}
We are interested in the case when the join $M_1\star_{k_1,k_2} M_2$ is a smooth manifold. This can be guaranteed by demanding the condition $\gcd(\upsilon_1k_2,\upsilon_2k_1)=1$ where $\upsilon_i$ is the order of the orbifold $\calb_i$.

\begin{theorem}\label{jointhm}
Let $\cald_{k_1,k_2}$ be the contact structure on the $(k_1,k_2)$-join $M_1\star_{k_1,k_2} M_2$ constructed as described above from the contact manifolds $(M_i,\cald_i)$. Let $\gn(\gro_i)$ denote the cardinality of the set $\gS\gC_T(\calb_i,\gro_i)$ of conjugacy classes of maximal tori in the group of Hamiltonian isotopies $\gH\ga\gm(\calb_i,\gro_i)$ of the base orbifolds $\calb_i$. Then on the join $M_1\star_{k_1,k_2} M_2$ the cardinality $\gn_R(\cald_{k_1,k_2})$ of the set of conjugacy classes of maximal tori of Reeb type in the contactomorphism group $\gC\go\gn(M_1\star_{k_1,k_2} M_2,\cald_{k_1,k_2})$ satisfies $\gn_R(\cald_{k_1,k_2})\geq \gn(\gro_1)\gn(\gro_2)$. In particular, $\cald_{k_1,k_2}$ has an $N$-bouquet of Sasaki cones where $N=\gn(\gro_1)\gn(\gro_2)$.
\end{theorem}

\begin{proof}
Let $T_i$ be a maximal torus in $\gH\ga\gm(\calb_i,\gro_i)$. Then the image of $T_1\times T_2$ under the natural subgroup inclusion
$$\gri:\gH\ga\gm(\calb_1,\gro_1)\times \gH\ga\gm(\calb_2,\gro_2)\ra{2.0} \gH\ga\gm(\calb_1\times \calb_2,\gro_{k_1,k_2}),$$
where $\gro_{k_1,k_2}=k_1\gro_1+k_2\gro_2$, is contained in a maximal torus $T$ in $\gH\ga\gm(\calb_1\times \calb_2,k_1\gro_1+k_2\gro_2)$. Furthermore, if $T_1$ and $T_1'$ are non-conjugate tori in $\gH\ga\gm(\calb_1,\gro_1)$, then $T_1\times T_2$ and $T_1'\times T_2$ are non-conjugate in $\gH\ga\gm(\calb_1\times \calb_2,\gro_{k_1,k_2})$ and similarly when 1 and 2 are interchanged. It follows that $\gn(\gro_{k_1,k_2})\geq \gn(\gro_1)\gn(\gro_2)$. The result then follows by applying Theorem \ref{contor}.
\end{proof}

More can be said in the toric case, for then $T_1\times T_2$ is not only maximal, but has maximal dimension. So 
\begin{corollary}\label{joincor}
The join of toric contact manifolds of Reeb type is a toric contact manifold of Reeb type. 
\end{corollary}

\section{Toric Contact Structures and Toric Symplectic Orbifolds}\label{torconsect}

\subsection{Toric Contact Structures}\label{torrev}
Toric contact structures were studied in detail in \cite{BM93}. They appear in two guises, those where the action of the torus is free, and those where it is not. When the dimension of $M$ is greater than $3$, the torus action is of Reeb type if and only if the action is not free and the moment cone contains no non-zero linear subspace. Moreover, it was shown in \cite{BG00b} that toric contact structures of Reeb type are all of Sasaki type. A complete classification of toric contact structures was given by Lerman in \cite{Ler02a} (the three dimensional case is somewhat special). The toric contact structures of Reeb type are classified by their moment cones when $n>1$ which are certain rational polyhedral cones called good cones by Lerman. Here is the construction.  Let $(M,\cald,T)$ be a contact manifold of dimension $2n+1$ with $n>1$ with a non-free action of an $(n+1)$-dimensional torus $T=T^{n+1}$. A polyhedral cone is said to be {\it good} if the annihilator of a linear span of a codimension $k$ face, i.e. a codimension $k$ subspace, is the Lie algebra of a subtorus $T^{n+1-k}$ of $T^{n+1},$ and the normals to the face form a basis of the integral lattice $\bbz_{T^{n+1-k}}$ of $T^{n+1-k}.$ Moreover, two such toric contact manifolds are isomorphic if and only if their polyhedral cones differ by a $GL(n+1,\bbz)$ transformation. Consider the moment map $\Upsilon:\cald^o_+\ra{1.6} \gt^*$ given by Equation (\ref{tormom}). The image of $\Upsilon$ is a good polyhedral cone $C(\cald,T).$ Now recall that the Sasaki cone $\gt^+(\cald,J)$ is the interior of the dual cone $C^*(\cald,T)$ to the moment cone $C(\cald,T)$. For every $\xi\in \grk(\cald,J)$ the intersection of the hyperplane $\eta(\xi)=1$ with the moment cone $C(\cald,T)$ is a simple convex polytope $P$, where a convex polytope is {\it simple} if there are precisely $n$ codimension one faces, called {\it facets}, meeting at each vertex.  Moreover, $\xi$ is quasi-regular if and only if it lies in the lattice $\ell$ of circle subgroups of $T,$ so $\xi$ is quasi-regular precisely when the polytope $P$ is rational \cite{BG00b}. In this case we let $p_i$ denote the outward primitive normal vector to the $ith$ facet. Then there is a positive integer $m_i$ such that $m_ip_i\in\ell$ for each facet. This give rise to a labelled (LT) polytope $P_{LT}$ of Lerman and Tolman \cite{LeTo97}. This procedure can be reversed. That is, starting with any LT-polytope $P_{LT}$ we construct a good rational polyhedral cone $C(P_{LT})$ as follows:
 \begin{equation}\label{pltcone}
C(P_{LT})=\{tP_{LT}~|~t\in [0,\infty)\}.
\end{equation}
\begin{proposition}\label{conepoly}
$C(P_{LT})$ is the good rational polyhedral cone associated to the K-contact manifold constructed from the almost K\"ahler orbifold $(\calz,\gro)$ by Theorem \ref{kconinversionthm}.
\end{proposition}

Following Lerman \cite{Ler02a,Ler03b} we have
\begin{definition}\label{torconequiv}
For $i=1,2$ two contact manifolds $(M_i,\cald_i,T_i)$ with an effective action of a torus $T_i$ are said to be {\bf $T$-equivariantly contactomorphic} or {\bf isomorphic} if there is a diffeomorphism $\varphi:M_1\ra{1.5} M_2$ and an isomorphism $\grg:T_1\ra{1.5} T_2$ such that $\varphi_*\cald_1=\cald_2$ and $\varphi(t\cdot x)=\grg(t)\cdot\varphi(x)$ for all $t\in T_1$ and all $x\in M_1.$ 
\end{definition}

In the case that $(M_1,\cald_1)=(M_2,\cald_2)=(M,\cald)$, we can identify $T_1$ and $T_2$ as subgroups of $\gC\go\gn(M,\cald)$ in which case this isomorphism just means that $T_1$ and $T_2$ are conjugate tori in $\gC\go\gn(M,\cald)$.

Identifying $T_1=T_2=T$ the isomorphism $\grg:T\ra{1.5} T$ of Definition \ref{torconequiv} induces its differential $\grg_*:\gt\ra{1.5} \gt$ which preserves the integral lattice $\bbz_T$. This in turn induces the `codifferential' $(\grg^{-1})^*:\gt^*\ra{1.5} \gt^*$ which preserves the dual weight lattice $\bbz_T^*={\rm Hom}(\bbz_T,\bbz).$ Otherwise stated we have $(\grg^{-1})^*\in {\rm GL}(\bbz_T^*).$ We conclude

\begin{proposition}\label{momconeequiv}
Let $\varphi:(M_1,\cald_1)\ra{1.6} (M_2,\cald_2)$ be a $T$-equivariant contactomorphism, and let $C(\cald_1,T_1)$ and $C(\cald_2,T_2)$ denote their respective moment cones. Then  $(\grg^{-1})^*(C(\cald_2,T_2))=C(\cald_1,T_1).$
\end{proposition}

Another rewording is

\begin{proposition}\label{momconeequiv2}
Let $T_1$ and $T_2$ be two abstractly isomorphic tori in $\gC\go\gn(M,\cald)$ and let $C(\cald_1,T_1)$ and $C(\cald_2,T_2)$ denote their respective moment cones. Suppose also that there does not exist a linear map $L\in {\rm GL}(\bbz_T^*)$ such that $L(C(\cald_1,T_1))=C(\cald_1,T_2).$ Then $T_1$ and $T_2$ are not conjugate in $\gC\go\gn(M,\cald)$.
\end{proposition}

Lerman \cite{Ler04} has shown how the topology of such toric compact contact manifolds is encoded in the polyhedral cone. First the odd dimensional Betti numbers vanish through half of the dimension, second $\pi_1(M)$ is isomorphic to the finite Abelian group $\bbz_T^{n+1}/\ell$ where $\ell$ is the sublattice of $\bbz_{T^{n+1}}$ generated by the normal vectors to the facets, and finally $\pi_2(M)=N-n-1$ where $N$ is the number of facets of the polyhedral cone.

The following is an immediate corollary of Lerman's classification theorem:
\begin{proposition}\label{lercor}
Let $(M^{2n+1},\cald,T)$ be a compact toric contact manifold of Reeb type with $n>1.$ Then $\gn_R(\cald,n+1)=\gn(\cald,n+1).$
\end{proposition}

Our main theorem concerning toric contact structures says that the map (\ref{Jmap4}) is injective for $\gr=n+1$ on toric contact manifolds of Reeb type.

\begin{theorem}\label{Qinj}
Let $(M,\cald,T)$ and $(M,\cald,T')$ be toric contact structures of Reeb type on a compact manifold $M$, and let $(\cald,J)$ and $\cald,J')$ be their underlying CR structures. Suppose further that $T$ and $T'$ are conjugate in $\gC\go\gn(M,\cald)$, so $\bar{\gQ}([J])=\bar{\gQ}([J'])$. Suppose also that $\eta$ is a contact form representing $\cald$ such that $\cals=(\xi,\eta,\Phi,g)$ and $\cals=(\xi,\eta,\Phi',g')$ are both K-contact with $\Phi|_\cald=J$ and $\Phi'|_\cald=J'$. Then $[J]=[J']$.
\end{theorem}

\begin{proof}
First since the structures are toric of Reeb type the almost complex structures can be assumed integrable by Theorem 5.2 of \cite{BG00b}. By perturbing the K-contact structure we can assume that $\eta$ is quasi-regular in which case the quotient $M/\Xi$ is a symplectic orbifold $(\calz,\gro)$ with symplectic form $\gro$ satisfying $\pi^*\gro=d\eta$ where $\Xi$ is the circle group generated by the Reeb vector field $\xi$. By Theorem \ref{contor} the tori $\grr(T)$ and $\grr(T')$ are conjugate in $\gH\ga\gm(\calz,\gro)$ on $(\calz,\gro)$. Now the transverse complex structures $J$ and $J'$ project to complex structures $\cJ$ and $\cJ'$ on $\calz$. So the toric K\"ahler orbifolds $(\calz,\gro,\cJ,T)$ and $(\calz,\gro,\cJ',T')$ are equivariantly isomorphic. So by \cite{LeTo97} their polytopes are isomorphic. In particular, their fans are isomorphic, so by \cite{LeTo97} they are biholomorphic implying that $\cJ$ and $\cJ'$ are equivalent complex structures. Thus, their lifts $J$ and $J'$ are equivalent transverse complex structures.
\end{proof}

\begin{remark}
Theorem \ref{Qinj} does not hold generally if the contact structure is not toric. As we shall see below in Section \ref{whs} there are many counterexamples when $\dim\grk(\cald,J)=1$.
\end{remark}

\subsection{Toric Contact Structures of Reeb Type in Dimension Five}
For dimension five there is much more detailed information about toric contact manifolds. There is the Barden-Smale classification of simply connected 5-manifolds, together with the work of Oh \cite{Oh83} that determined which Barden-Smale manifolds admit an effective $T^3$-action. These are precisely the simply connected 5-manifolds without torsion, namely, $S^5$, $k$-fold connected sums of $S^2\times S^3$, the non-trivial $S^3$-bundle over $S^2$, and the connected sum of the latter with $k$-fold connected sums of $S^2\times S^3$. An explicit construction for certain toric contact structures on these manifolds was given in \cite{BGO06} where it was shown that they are regular and fiber over the blow-ups of Hirzebruch surfaces. All of these can be obtained by reduction and a classification of all such toric structures from this point of view is currently under study. A starting point for this procedure is described in Proposition 11.4.3 of \cite{BG05}. Here I present several examples that describe the configurations of toric Sasakian structures. 

\begin{example}\label{wzex} There are exactly two $S^3$-bundles over $S^2$. The trivial bundle $S^2\times S^3$ and the non-trivial bundle denoted by $X_\infty$ in Barden's notation. I consider two types of contact structures $\cald_{k_1,k_2}$ and $\tcald_{l,e}$ both labelled by a pair of relatively prime integers and whose first Chern classes satisfy $c_1(\cald_{k_1,k_2})\neq 0$ and $c_1(\tcald_{l,e})\neq 0$. The details of these contact structures, which are based on \cite{Kar03} and \cite{Ler03b}, are presented in \cite{Boy10b} where the existence of extremal Sasakian structures is given. The contact structures $\cald_{k_1,k_2}$ all exist on $S^2\times S^3$, whereas, $\tcald_{l,e}$ exists on $S^2\times S^3$ if $l$ is even and on $X_\infty$ if $l$ is odd. We also assume\footnote{The ratio $\frac{k_1}{k_2}$ is incorrectly written as $\frac{k_2}{k_1}$ in both \cite{BGO06} and \cite{BG05}. It is correct, however, in \cite{Boy10b}.} $k_1>k_2>0$ and $l>e>0$ and put $N(k_1,k_2)=\lceil \frac{k_1}{k_2} \rceil$ and $N(l,e)= \lceil \frac{e}{l-e} \rceil$ where $ \lceil r\rceil$ denotes the smallest integer greater than or equal to $r$. For these contact structures there is an $N$-bouquet of Sasakian structures based at $\xi$, and each cone of the bouquet has an open set of extremal Sasakian metrics. In each case exactly one cone of the bouquet has a positive Sasakian structure, and the remainder of the cones have indefinite Sasakian structures. Moreover, it is known \cite{Pat09} that the contact structures $\cald_{k_1,k_2}$ are all inequivalent. Notice that in this case $\cald_{k_1,k_2}$ is the contact structure obtained from the $(k_1,k_2)$-join $S^3\star_{k_1,k_2}S^3$ with the standard contact structure on each $S^3$. The contact structures $\tcald_{l,e}$ and $\tcald_{l',e'}$ are inequivalent if $l-2e\neq l'-2e'$ since their first Chern classes are different, but otherwise the equivalence problem is not known. Similarly, if $l=2k$ and $k_1-k_2\neq k-e$, then $\cald_{k_1,k_2}$ and $\tcald_{2k,e}$ are inequivalent.
\end{example}

It is straightforward to construct many more examples of Sasaki bouquets on the $k$-fold connected sum $k(S^2\times S^3)$ by representing them as circle bundles over symplectic blow-ups of $\bbc\bbp^2$ and $\bbc\bbp^1\times \bbc\bbp^1$. Indeed, in \cite{Kar99} Karshon gives a classification of all Hamiltonian $S^1$ actions on compact symplectic 4-manifolds in terms of certain graphs. Moreover, she determines which ones extend to a $T^2$ toric action. So one can get examples of Sasaki bouquets with Sasaki cones of different dimensions. Here is one such example consisting of two Sasaki cones one of which is toric and the other which is not.

\begin{example}\label{vardimex1}
We consider the 5-manifold $3(S^2\times S^3)$ with contact structure $\cald$ obtained from the Boothby-Wang construction over the symplectic 4-manifold obtained from an equal size $\gre$-symplectic blow-up of $\bbc\bbp^2$ at 3 points. In Example 3.5 of \cite{KaKePi07} it is shown that by choosing $\gre$ sufficiently small the symplectomorphism group of such a symplectic manifold contains at least two maximal tori, one of dimension one, and the other of dimension two. Thus, by Theorem \ref{cardReeb} the contact structure $\cald$ has a Sasaki 2-bouquet with one Sasaki cone of dimension 3 and another of dimension 2. Both of the associated almost complex structures are integrable, the latter by Theorem 7.1 of \cite{Kar99}; hence, the contact metric structures of the bouquet are all Sasakian. 
\end{example}

\begin{example}
The next example is due mainly to \cite{FOW06,CFO07}. These contact structures live on the $k$-fold connected sums $k(S^2\times S^3)$, and all satisfy  $c_1(\cald)=0$. For each positive integer $k$ there are infinitely many contact structures that have a $1$-bouquet of Sasakian structures, that is a single Sasaki cone of Sasakian structures. Moreover, each cone has a Sasaki-Einstein metric, hence by \cite{BGS06}, an open set of extremal Sasakian metrics. Here all Sasakian structures are positive. Recently the contact equivalence problem \cite{BoPa10,Boy11} has been studied for the case $k=1$. For example, it is shown that a subclass of toric contact structures $Y^{p,q}$ with $c_1(\cald)=0$ on $S^2\times S^3$ first studied by physicists \cite{GMSW04a} come in bouquets of Sasakian structures belonging to the same contact structure when $p$ is fixed and $q$ is any positive integer relatively prime to $p$ and less than $p$. 

\end{example}

\subsection{Toric Contact Structures of non-Reeb Type}
In the case of a toric contact structure that is not of Sasaki type on a manifold $M$ of dimension greater than three and the action of the torus $T$ is free, $M$ is the total space of a principal $T^{n+1}$-bundle over $S^n$. It was shown in \cite{LeSh02} that $T^{n+1}$ acts freely as contact transformations preserving the standard contact structure on the unit sphere bundle $S(T^*T^{n+1})$.   Such contact manifolds seem to have first been studied in detail by Lutz \cite{Lut79}. By Lerman's classification theorem \cite{Ler02a} there is a unique $T^{n+1}$-invariant contact structure on each principal bundle, so we have
\begin{theorem}\label{freetoric}
Let $M$ be a $2n+1$-dimensional manifold with $n>1$ and a free action of an $(n+1)$-dimensional torus $T^{n+1}.$ Then on each principal $T^{n+1}$-bundle $P$ we have $\gn_R(\cald)=0$ for any contact structure $\cald$ and $\gn(\cald,n+1)=1.$
\end{theorem}

The principal $T^{n+1}$ bundles over $S^n$ are classified by $H^2(S^n,\bbz^{n+1}),$ so if $n>2$ then $P$ is the trivial bundle $T^{n+1}\times S^n.$ In the case $n=2$ we have $H^2(S^2,\bbz^{3})=\bbz^3.$

\begin{example}\label{unitsph}
We construct the natural contact metric structure on the unit sphere bundle $S(T^*T^{n+1})$ by representing it as the unit sphere in $T^*T^{n+1}$ minus the $0$ section which in the standard coordinates $(p_i,x^i)_{i=0}^n$ on the cotangent bundle is given by $\sum_{i=0}^np_i^2=1$. The canonical 1-form on $T^*T^{n+1}$  is $\theta=\sum_{i=0}^np_idx^i,$ so the contact form on $S(T^*T^{n+1})$ is defined by $\eta=\theta|_{S(T^*T^{n+1})}.$ We can cover $S(T^*T^{n+1})$ by `charts' with coordinates $(x^0,x^1,\ldots,x^n)$ on $T^{n+1}$ and coordinates in the fibre $(p_1,\ldots,p_n)$. It is convenient to use some vector notation by writing $\bfx=(x^1,\ldots,x^n)$ and $\bfp=(p_1,\ldots,p_n).$ Then we have $p_0=\pm \sqrt{1-|\bfp|^2}$, so the contact form $\eta$ is given in these coordinates by 
\begin{equation}\label{1formunitsph}
 \eta_\pm=\pm \sqrt{1-|\bfp|^2}dx^0 +\bfp\cdot d\bfx.
\end{equation}
We shall just write $\eta$ without specifying the sign of $p_0.$ It is clear that $\eta$ is invariant under the standard action of the torus on itself, $x^i\mapsto x^i+a^i.$ One can easily check that the Reeb vector field is 
\begin{equation}\label{unitsphReeb}
\xi=\frac{1}{\sqrt{1-|\bfp|^2}}\partial_{x^0}-\frac{1}{1-|\bfp|^2}\bfp\cdot \partial_\bfp,
\end{equation}
and the contact bundle is 
\begin{equation}\label{conbununitsph}
\cald={\rm span}\{\partial_\bfp, \partial_\bfx -\frac{\bfp}{\sqrt{1-|\bfp|^2}}\partial_{x^0}\}.
\end{equation}
For convenience we set ${\boldsymbol R}= \partial_\bfx -\frac{\bfp}{\sqrt{1-|\bfp|^2}}\partial_{x^0}$ and define an almost complex structure $J$ on $\cald$ by $J\bfR=\partial_\bfp$ and $J\partial_\bfp=-\bfR.$ We extend $J$ uniquely to the endomorphism $\Phi$ by demanding $\Phi\xi=0.$ If we let $\bfS=\frac{d\bfx}{1-|\bfp|^2}+\frac{\bfp dx^0}{\sqrt{1-|\bfp|^2}}$, the dual of $\bfR$, then
\begin{equation}\label{Phiunitsph}
\Phi =\partial_\bfp\otimes\cdot \bfS-\bfR\otimes\cdot d\bfp -\frac{\bfp}{1-|\bfp|^2}\cdot(\partial_\bfp-\bfR)\otimes \eta.
\end{equation}
\end{example}

\begin{example}\label{t3} {\it Tight Contact Structures on $T^3$.}
The 3-torus $T^3$ is a toric contact manifold \cite{Ler02a} which is not of K-contact type. (If a $2n+1$-dimensional manifold $M$ admits a K-contact structure, its cuplength must be less than or equal to $2n$ \cite{Ruk93,BG05}). On the other hand, Giroux \cite{Gir94,Gir99} has shown that the contact structures $\cald_k$ defined by the 1-forms $\eta_k=\cos k\theta dx_1 +\sin k\theta dx_2$ for different $k\in\bbz^+$ are inequivalent, where $(x_1,x_2,\theta)$ denote the Cartesian coordinates on $T^3.$ Moreover, they are all the tight positively oriented contact structures \cite{Kan97,Gir00,Hon00b} on $T^3$, up to contactomorphism, and all are symplectically fillable; however, Eliashberg \cite{Eli96} showed that only $\cald_1$ is holomorphically fillable. For each $\eta_k\in\gC^+(\cald_k)$ the 2-torus $T^2$ generated by $\partial_{x_1},\partial_{x_2}$ is the unique maximal torus \cite{Ler02a}, up to conjugacy in $\gC\go\gn(T^3,\cald_k)$ for each $k\geq 1$, but it cannot be of Reeb type. So $\gn_R(\cald_k)=0$ and $\gn(\cald_k,2)=1$ for all $k.$ By taking 
$$\Phi_k =\frac{1}{k}(\cos k\theta~\partial_\theta\otimes dx_2-\sin k\theta~\partial_\theta\otimes dx_1) +k(\sin k\theta~\partial_{x_1}\otimes d\theta -\cos k\theta~\partial_{x_2}\otimes d\theta)$$
it is easy to see the contact metric (\ref{conmet}) is just the flat metric on $T^3$ for each $k$. 
\end{example}

\begin{example}\label{overtw} {\it Contact Structures on $S^3$}.
Eliashberg \cite{Eli89} (see also Chapter 8 of \cite{ABKLR94}) proved that up to oriented contactomorphism there are precisely two co-oriented contact structures on $S^3$ associated with the two-plane bundle whose Hopf invariant vanishes, the standard tight contact structure which is of Sasaki type, and an overtwisted contact structure. Furthermore, since Sasakian structures are fillable, hence tight, overtwisted structures are not Sasakian. Lerman \cite{Ler01} constructs overtwisted contact structures on $S^3$ using his method of contact cuts. Locally the geometry resembles that of Example \ref{t3}. We have the contact form $\eta_k=\cos(2k+\frac{1}{2})\pi t)d\theta_1 +\sin(2k+\frac{1}{2})\pi t)d\theta_2,$ with Reeb vector field $\xi_k= \cos(2k+\frac{1}{2})\pi t)\partial_{\theta_1} +\sin(2k+\frac{1}{2})\pi t)\partial_{\theta_2}.$ The contact bundle $\cald_k$ is spanned by 
$$\{\partial_t,\cos((2k+\frac{1}{2})\pi t)\partial_{\theta_2}-\sin((2k+\frac{1}{2})\pi t)\partial_{\theta_1}\}.$$ 

Now $\cald_0$ is the standard tight contact structure while the $\cald_k$ for $k>0$ are overtwisted. However, Eliashberg \cite{Eli89} proved that the contact structures $\cald_k$ are contactomorphic for all $k\in \bbz^+$. We denote this isomorphism class of contact structures by $\tcald$. Lerman \cite{Ler01} proved that the $\cald_k$s are not $T^2$-equivariantly contactomorphic. It follows that there are a countable infinity of non-conjugate tori on the overtwisted contact structure $\tcald$ on $S^3.$ Moreover, as in Example \ref{t3} these overtwisted contact structures are toric and as toric contact structures they are classified by a countable set \cite{Ler02a}, so $\gn_R([\cald])=0,$ and $\gn([\cald],2)=\aleph_0.$ In \cite{Ler01} Lerman also shows that a similar result holds for lens spaces. It would be interesting to find a preferred compatible contact metric (transversally complex structure) for the overtwisted structures.

\end{example}

\section{Further Examples of Reeb Type}\label{moreex}

In this section I describe some non-toric examples of contact structures  that exhibit maximal tori. There  are several methods to construct such examples. In the case of contact structures of Sasaki type, one way is to consider links of hypersurface singularities of weighted homogeneous polynomials. This method has proven to be of great advantage for constructing Einstein metrics in odd dimensions, see \cite{BG05} and references therein. The second method that constructs contact structures of K-contact or Sasaki type  begins with a symplectic orbifold (manifold) and constructs certain orbibundles whose total space has a K-contact or Sasakian structure. In these cases we always have maximal tori of Reeb type. 

\subsection{Weighted Hypersurface Singularities}\label{whs}
I now briefly describe the first method. Recall  
that a polynomial $f\in {\mathbb C}[z_0,\ldots,z_n]$ is said to be a 
{\it weighted homogeneous polynomial} of {\it degree} $d$ and 
{\it weight} ${\bf w}=(w_0,\ldots,w_n)$ if for any $\lambda \in 
{\mathbb C}^*=\bbc\setminus \{0\},$ we have
$$
f(\lambda^{w_0} z_0,\ldots,\lambda^{w_n}
z_n)=\lambda^df(z_0,\ldots,z_n)\, .
$$
We are interested in those weighted homogeneous polynomials $f$ whose 
zero locus in ${\mathbb C}^{n+1}$ has only an isolated singularity at the 
origin. This assures us that the {\it link} $L_f({\bf w},d)$ defined as 
$f^{-1}(0)\cap S^{2n+1}$ is a smooth manifold, where $S^{2n+1}$ is the 
$(2n+1)$-sphere in $\bbc^{n+1}$. By the Milnor fibration theorem 
\cite{Mil68}, $L= L_f({\bf w},d)$ is a closed $(n-2)$-connected manifold that 
bounds a parallelizable manifold with the homotopy type of a bouquet of $n$-spheres. Furthermore, $L_f({\bf w},d)$ admits a quasi-regular 
Sasakian structure ${\oldmathcal S}=(\xi_{\bf w},\eta_{\bf w},\Phi_{\bf w},g_{\bf w})$ in a 
natural way \cite{Abe77,Tak,Var80}, thus defining a contact structure $\cald_{\bfw,d}$. Moreover, $L_f({\bf w},d)$ is the total space of an $S^1$ orbibundle over a projective algebraic orbifold $\calz_{\bfw,d}$ embedded (as orbifolds) in the weighted projective space $\bbc\bbp(\bfw)$. Both the orbifold $\calz_{\bfw,d}$ and the link $L_f({\bf w},d)$ come with a family of integrable complex structures. The corresponding family of transverse complex structures on $L_f({\bf w},d)$ is denoted by $\calj_{\bfw,d}$. This is a finite dimensional subspace of $\calj(\cald_{\bfw,d})$ of complex dimension $h^0\bigl(\bbc\bbp(\bfw),\calo(d)\bigr)-n-1.$ These complex structures are generally inequivalent and the complex dimension of the infinitesimal moduli space $\calm_{\bfw,d}$ can be easily seen (cf. Proposition 5.5.6 of \cite{BG05}) to be 
\begin{equation}\label{dimmod}
\dim\calm_{\bfw,d}=h^0\bigl(\bbc\bbp(\bfw),\calo(d)\bigr)-\sum_ih^0\bigl(\bbc\bbp(\bfw),\calo(w_i)\bigr)+\dim~\gA\gu\gt(\calz_{\bfw,d}).
\end{equation}

\begin{proposition}\label{whscomp}
Assume that $2w_i< d$ for all but at most one weight. Then for each complex structure $J\in \calj_{\bfw,d}$ the torus $T$ in the conjugacy class of $\calc_T(J)$ associated to $J$ equals $\Xi$ and is maximal in the contactomorphism group $\gC\go\gn(L,\cald_{\bfw,d}).$ Moreover, $\gn_R(\cald_{\bfw,d})=1.$
\end{proposition}

\begin{proof}
By Lemma 5.5.3 of \cite{BG05} the complex automorphism group $\gA\gu\gt(\calz_{\bfw,d})$ of $\calz_{\bfw,d}$ is discrete, so by Proposition \ref{maxtorconsym} there are no transverse holomorphic tori. This implies that the only torus leaving any of the transverse complex structures in $ \calj_{\bfw,d}$ invariant is $\Xi$, and this leaves them all invariant and is maximal in $\gC\go\gn(L,\cald_{\bfw,d})$ by Theorem \ref{contor}. 
\end{proof}

\begin{example}\label{whsex1}
When the hypothesis of Proposition \ref{whscomp} holds, $\dim\gA\gu\gt(\calz_{\bfw,d})=0$ and $\dim\grk(\cald,J)=1$. There are many examples in \cite{BGN03c,BG03,BGK05,BG06b,BG05,BN09} with $\dim\calm_{\bfw,d}>0$ showing that the map $\bar{\gQ}$ defined by Equation (\ref{Jmap2}) is not injective when $\dim\grk(\cald,J)=1$. It is interesting to note that many of these occur on homotopy spheres (cf. \cite{BGK05,BG05}) where the dimension of $\calm_{\bfw,d}$ can be  quite large. For example, consider the weighted hypersurface $L_f(\bfw,d)$ where the degree $d$ and weights $\bfw=(w_0,\cdots,w_7)$ are given by $w_i=\frac{d}{a_i},~d=\lcm(a_0,\cdots,a_6,a_6-2)$ where $a_i$ is determined by the recursion relation $a_k=a_0\cdots a_k +1$ beginning with $a_0=2$. Note that the $a_i$s grow with double exponential growth. For example, $a_5=10650056950807$. In this case $L_f(\bfw,d)$ is diffeomorphic to $S^{13}$ and has $\dim\calm_{\bfw,d}>2\times 10^{13}$, so the fibre of the map $\bar{\gQ}_1:\calj(\cald_{\bfw,d})\ra{1.5} \gS\gC_T(\cald_{\bfw,d},1)$ has dimension $>2\times 10^{13}$.
\end{example}

Examples of links of weighted homogeneous polynomials with $\dim\grk(\cald,J)>1$ are easily obtained when the hypothesis of Proposition \ref{whscomp} fails. Any Brieskorn polynomial with more than one quadratic monomial will suffice. 

\begin{example}\label{whsex2}
The links obtained from the Brieskorn polynomial 
$z_0^{6k-1}+z_1^3+z_2^2+z_3^2+z_4^2$ represent Milnor's 28 homotopy 7-spheres as $k=1,\cdots,28$. This polynomial has degree $d=6(6k-1)$ and weight vector $\bfw(k)=(6,2(6k-1),3(6k-1),3(6k-1),3(6k-1))$. The automorphism group $\gA\gu\gt(\calz_{\bfw(k),d})$ is $O(3,\bbc)$. The Sasaki cone can easily be computed by choosing an $S^1$ subgroup of $O(n,\bbr)$. For example, we take the circle generated by the vector field $X=i\bigl(z_3\partial_{z_2}-z_2\partial_{z_3}- \bz_3\partial_{\bz_2}+\bz_2\partial_{\bz_3}\bigr)$ 
restricted to the link. For each positive integer $k$ the Lie algebra $\gt$ of the maximal torus is spanned by $\{\xi_{\bfw(k)},X\}$. So the positivity condition $\eta_{\bfw(k)}(a\xi_{\bfw(k)} +bX)>0$ gives a two dimensional Sasaki cone determined by $-2a<b<2a$. In this case we find $\dim_\bbc\calj_{\bfw(k),d}=3$ and $\dim_\bbc\calm_{\bfw(k),d}=0$. So any $T$-equivalent complex structures in $\calj_{\bfw(k),d}$ are actually equivalent as complex structures.
\end{example}

\subsection{Other Results in Dimension Five} Dimension five is particularly amenable to study for several reasons. First, there is the Barden-Smale classification of all simply connected compact 5-manifolds. Second, there is an extremely rich knowledge of compact symplectic 4-manifolds from which to construct contact 5-manifolds by the Boothby-Wang construction.

Generally, one can apply the Boothby-Wang construction to the recent finiteness result of \cite{KaKePi07} which says that any compact symplectic 4-manifold admits only finitely many inequivalent toric actions to obtain:
\begin{proposition}\label{KKPprop}
Let $\cald$ be a regular contact structure on a compact five dimensional manifold $M$. Let $\xi$ be the Reeb vector field of the regular contact 1-form $\eta$. Then any bouquet of Sasaki cones based at $\xi$ has finite cardinality.
\end{proposition}
 
It would be interesting to see whether the finiteness result in \cite{KaKePi07} extends to compact symplectic orbifolds. If so then one could strengthen Proposition \ref{KKPprop} to all toric contact structures of Reeb type in dimension five. 

One can easily obtain Sasakian structures on non-simply connected 5-manifolds by applying the Boothby-Wang construction to non-simply connected compact 4-manifolds. Of particular interest are ruled surfaces over Riemann surfaces of arbitrary genus $g$. It is known that extremal K\"ahler metrics exist for certain K\"ahler classes on these manifolds \cite{To-Fr98} (see also \cite{Gau09b} and references therein). The Boothby-Wang 5-manifolds over these surfaces will have Sasakian structures with $\dim\grk(\cald,J)=2$ and at least some will have extremal Sasakian metrics. It would be interesting to see whether bouquets can occur in this case. A recent related result appears in \cite{Noz09} where it is shown that a compact K-contact 5-manifold with $\dim\grk(\cald,J)=2$ must have a compatible Sasakian structure. This generalizes a result of Karshon in \cite{Kar99} to certain cyclic orbifolds.

\subsection{Higher Dimension}\label{vardim} 
Higher dimensional examples are easy to construct by the methods discussed previously. In particular,
one efficient way to obtain bouquets with cones of varying dimension is by considering circle bundles over the so-called polygon spaces \cite{Ler02b}.

\begin{example}\label{htex}
The polygon spaces ${\rm Pol}(\gra)$ are defined as the quotient of
\begin{equation}\label{polsp}
\widetilde{{\rm Pol}}(\gra)=\{(\bfx_1,\ldots,\bfx_m)\in (\bbr^3)^m~|~|\bfx_i|=\gra_i,~\sum_{i=1}^m\bfx_i=0\}.
\end{equation}
by the diagonal action of $SO(3)$ on $(\bbr^3)^m$. If the equation $\sum_{i=1}^m\varepsilon_i\gra_i=0$ has no solution with $\varepsilon_i=\pm 1$, then ${\rm Pol}(\gra)$ is a simply connected smooth compact symplectic manifold of dimension $2(m-3)$. We shall always assume this condition. It is known \cite{HaKn98} that the cohomology groups $H^k({\rm Pol}(\gra),\bbz)$ vanish when $k$ is odd, and their diffeomorphism type can be determined from the combinatorial data. For a more detailed background on these spaces refer to \cite{HaKn98} and references therein. In \cite{HaTo03} Hausmann and Tolman study the maximal tori in polygon spaces. They construct a sequence ${\rm Pol}(1,1,2,2,3,3,3,\frac{1}{2},\cdots, \frac{1}{2^m})$ of polygon spaces of dimension $2(m+4)$ which have 3 conjugacy classes of Hamiltonian tori of dimension $m+2,m+3,m+4$, respectively. In particular, $m=0$ is a heptagon space ${\rm Pol}(1,1,2,2,3,3,3)$ which is an 8 dimensional simply connected symplectic manifold with maximal tori of dimensions, 2, 3, and 4. Following the prescription given in \cite{HaKn98} we compute the Poincar\'e polynomial for space to be
$$P_{{\rm Pol}(1,1,2,2,3,3,3)}=1+7t^2+15t^4+7t^6+t^8,$$
and its Hirzebruch signature is $3$.

Thus, as was observed by Lerman \cite{Ler02b} there is a 9-dimensional simply connected contact manifold $M$ with maximal tori of dimensions 3,4, and 5 in its contactomorphism group. More explicitly,

\begin{theorem}\label{polthm}
There is a compact simply connected spin manifold $M^9$ of dimension $9$ with Poincar\'e polynomial 
$$P_{M^9}=1+6t^2+8t^4+8t^5+6t^7+t^9$$
that admits a contact structure $\cald$ with non-vanishing first Chern class and a $3$-bouquet  consisting of Sasaki cones of dimensions 3, 4, and 5.
\end{theorem}

\begin{proof}
First, from Lerman's observation and our construction we get the $3$-bouquet.
By theorem 6.4 of \cite{HaKn98} the cohomology ring of the 8-dimensional polygon space $P=P_{{\rm Pol}(1,1,2,2,3,3,3)}$ has form 
$$H^*(P,\bbz)=\bbz[R,V_1,V_2,V_3,V_4,V_5,V_6]/\cali_P$$
where $\cali_P$ is a certain ideal whose quadratic elements are easily seen to be 
$$V_i^2+RV_i, (i=1,\cdots,6),~V_4^2+V_5^2+V_6^2-R^2,~V_3V_5,V_3V_6,V_4V_5,V_4V_6, V_5V_6.$$
In terms of these generators the class of the symplectic form $\gro$ on $P$ is 
$$[\gro]=9R +2(V_1+V_2)+4(V_3+V_4)+6(V_5+V_6).$$
This defines a primitive class in $H^2(P,\bbz)$, so we can construct the circle bundle $\pi:M^9\ra{1.7}P$ corresponding to this class whose total space is simply connected. Moreover, the Boothby-Wang construction gives a contact structure with contact 1-form $\eta$ satisfying $d\eta=\pi^*\gro$. Now using \cite{HaKn98} the first Chern class of $P$ is $c_1(P)=5R +2\sum_{i=1}^6V_i$. Denoting the pullback of the classes $R,V_i$ to $M^9$ by $\tR,\tV_i$ and using the relation $\pi^*[\gro]=0$ we have $c_1(\cald)=-4\tR-2(\tV_3+\tV_4)-4(\tV_5+\tV_6)$ which is non-vanishing and whose mod 2 reduction vanishes. The latter implies that $M^9$ is spin.

To compute the Poincar\'e polynomials we consider the Serre spectral sequence of the circle bundle $\pi:M_9\ra{1.7} P$. We know that the differential $d:E^{1,0}_2\ra{1.6} E_2^{0,2}$ satisfies $du=[\gro]$ where $u$ is the generator of $H^1(S^1,\bbz)$. We then get $d:E_2^{1,2}\ra{1.7} E_2^{0,4}$ as $d(u\otimes R)=[\gro]\cup R$, and $d(u\otimes V_i)=[\gro]\cup V_i$. It is not difficult to see by using the quadratic relations that this $d$ is injective implying that $H^3(M^9,\bbz)=0$ and that $H^4(M^9,\bbq)=\bbq^8$. The gives the first 3 terms in $P_{M^9}$ and rest of the terms follow from Poincar\'e duality. This proves the result.
\end{proof}

\begin{remark}
Since $H^2(P,\bbz)$ is free Abelian \cite{HaKn98} we see that $H^2(M^9,\bbz)$ is also free Abelian. However, $H^4(M^9,\bbz)$ contains torsion as can be seen by the equation 
$$d(u\otimes (V_4-V_3))=2(V_1+V_2)(V_4-V_3).$$
\end{remark}

\begin{remark}
Generally, the contact metric structures in the bouquet are K-contact. We do know that the ones belonging to the Sasaki cone of dimension 5 are Sasakian since this structure is toric. However, it is not known at this time whether the transverse almost complex structures of the lower dimensional Sasaki cones are integrable.
\end{remark}

\end{example}

One can easily apply the join construction and Theorem \ref{jointhm} to obtain many new examples of such behavior in higher dimension.

\newcommand{\etalchar}[1]{$^{#1}$}
\def\cprime{$'$} \def\cprime{$'$} \def\cprime{$'$} \def\cprime{$'$}
  \def\cprime{$'$} \def\cprime{$'$} \def\cprime{$'$} \def\cprime{$'$}
  \def\cdprime{$''$} \def\cprime{$'$} \def\cprime{$'$} \def\cprime{$'$}
  \def\cprime{$'$}
\providecommand{\bysame}{\leavevmode\hbox to3em{\hrulefill}\thinspace}
\providecommand{\MR}{\relax\ifhmode\unskip\space\fi MR }
% \MRhref is called by the amsart/book/proc definition of \MR.
\providecommand{\MRhref}[2]{%
  \href{http://www.ams.org/mathscinet-getitem?mr=#1}{#2}
}
\providecommand{\href}[2]{#2}

\end{document}